\documentclass[a4paper,11pt]{amsart}
\usepackage{amsfonts,amssymb,amsmath,latexsym}

\newcommand\cB{\mathcal{B}}

\newcommand\cL{\mathcal{L}}

\newcommand\cP{\mathcal{P}}

\newcommand\cS{\mathcal{S}}

\newcommand\cO{\mathcal{O}}

\newcommand{\CC}{\mathbb{C}}

\newcommand{\TT}{\mathbb{T}}

\newcommand{\DD}{\mathbb{D}}

\newcommand{\NN}{\mathbb{N}}
\newcommand{\sz}{\mathrm{(S)}}

\newcommand{\bsl}{\backslash}
\newcommand{\ovl}{\overline}
\newcommand{\ti}{\tilde}
\newcommand{\dsp}{\displaystyle}

\allowdisplaybreaks
\numberwithin{equation}{section}

\newtheorem*{theorem1}{Theorem 1}
\newtheorem*{theorem2}{Theorem 2}
\newtheorem*{theorem3}{Theorem 3}
\newtheorem*{theorem1'}{Theorem 1'}
\newtheorem*{theorem1''}{Theorem 1''}
\newtheorem*{theorem2'}{Theorem 2'}
\newtheorem*{theorem3'}{Theorem 3'}
\newtheorem*{theoremm}{Theorem} 

\newtheorem{theorem}{Theorem}[section]
\newtheorem{proposition}[theorem]{Proposition}
\newtheorem{lemma}[theorem]{Lemma}
\newtheorem{corollary}[theorem]{Corollary}
\newtheorem{definition}[theorem]{Definition}

\newtheorem{remark}[theorem]{Remark}

\DeclareMathOperator\supp{supp}

\begin{document}

\title[Multipoint Schur approximation]{Multipoint Schur algorithm and orthogonal rational functions: convergence properties.}
\author[L. Baratchart, S. Kupin, V. Lunot, M. Olivi]{L. Baratchart, S. Kupin, V. Lunot, M. Olivi}

\address{2004 route des Lucioles - BP 93
FR-06902 Sophia Antipolis Cedex, France}
\email{Laurent.Baratchart@sophia.inria.fr}

\address{IMB, Universit\'e Bordeaux 1, 351 cours de la Lib\'eration, 33405 Talence Cedex France}
\email{skupin@math.u-bordeaux1.fr}

\address{2004 route des Lucioles - BP 93
FR-06902 Sophia Antipolis Cedex, France}
\email{vlunot@gmail.com}

\address{2004 route des Lucioles - BP 93
FR-06902 Sophia Antipolis Cedex, France}
\email{Martine.Olivi@sophia.inria.fr}

\date{January, 25, 2010}

\thanks{This work was partially supported by grants  ANR-07-BLAN-024701 and  ANR-09-BLAN-005801}

\keywords{Approximation by rational functions, Schur algorithm, Schur (Verblunsky) parameters, orthogonal rational functions (orthogonal polynomials),  Wall orthogonal functions (Wall polynomials)}
\subjclass{Primary:  30B70. Secondary: 41A20}

\begin{abstract}
Classical Schur analysis is intimately connected to the theory of orthogonal 
polynomials on the circle \cite{Simon1}. We investigate here the connection 
between multipoint Schur analysis and orthogonal rational functions.
Specifically, we study the convergence of the  
Wall rational functions {\it via} the development of a rational analogue 
to the Szeg\H o theory,
in the case where the interpolation points may accumulate on the unit circle.
This leads us to generalize results from \cite{Khrushchev2001, Bultheel1999},
and yields asymptotics of a novel type. 
\end{abstract}

\maketitle

\section*{Introduction}\label{intro}
The theory of orthogonal polynomials, with respect to a positive measure
on the line or the circle, currently undergoes a period of 
intensive growth. To hint at recent advances, let us quote 
the papers by  Killip-Simon \cite{ks1}, Mart\'inez-Finkelstein et al.
\cite{MFMLS}, Mi$\tilde{\mbox{n}}$a-D\`iaz \cite{Mina},
Kuijlaars et al. \cite{kui}, 
McLaughlin-Miller, \cite{McLM},
Lubinsky \cite{lu2} and Remling \cite{re1}. 
A comprehensive account of many late developments in the field can be found 
in the monograph by Simon \cite{Simon1}. Let us mention in passing that, 
over the same period, non-Hermitian orthogonality with respect to complex 
measures, which is intimately connected with rational 
approximation and interpolation, made some progress too; see, for example,
Aptekarev \cite{Apt02}, Aptekarev-Van Assche \cite{AVA04}, Baratchart-K\"ustner-Totik \cite{BKT} and  Baratchart-Yattselev \cite{BaYa3}.

The connection between orthogonal polynomials on the unit circle
and the Schur algorithm is an old one.
Recall that a Schur function is an analytic map from the 
open unit disk into itself. The Schur algorithm, introduced by Schur 
and Nevanlinna \cite{Schur,nev}, associates to every Schur function a 
sequence of complex numbers of modulus at most one, called its Schur (or Verblunsky)
parameters. Since the {\it mise en sc\`ene} of the present article unfolds mainly in the framework of the Schur analysis, we shall call the parameters ``Schur", although the term ``Verblunsky'' seems to be more fair historically, see
Simon \cite[Sect. 1.1]{Simon1} for a discussion. These parameters may be viewed as hyperbolic analogues of the Taylor coefficients at the origin. They generate a 
continued-fraction expansion of the function, 
whose truncations give rise to the 
so-called Schur approximants. These are hyperbolic counterparts of the Taylor 
polynomials, see definition \eqref{e02} to come. Now, an elementary linear 
fractional transformation puts Schur functions in one-to-one correspondence 
with Carath\'eodory functions, {\it i.e.} 
analytic functions with positive 
real part in the disk, which are themselves in
bijection with positive measures on the circle {\it via}
the Herglotz transform.
A long time ago already, Geronimus and Wall observed the remarkable identity
 between the Schur parameters of a function and the recurrence coefficients
of the orthogonal polynomials associated to the corresponding measure
\cite{Geronimus,wa1}.
However, only relatively recently was it stressed
by Khrushchev  \cite{Khrushchev2001, khr2}
how properties of the measure, that govern the convergence of the
corresponding orthogonal polynomials, 
are linked to the convergence of the
Schur approximants \emph{on} the unit circle.

It must be pointed out that the Schur algorithm is among the seldom 
procedures preserving the Schur character in rational approximation;
equivalently, it yields Carath\'eodory 
rational approximants to Carath\'eodory functions on the disk 
or the half-plane. This feature is of fundamental importance in 
several areas of Physics and Engineering, where 
the Schur or Carath\'eodory nature of a transfer function is to be 
interpreted as a  passivity property of the underlying system. 
Moreover, in  such modeling issues, the relevant norms 
take place on the boundary of the 
analyticity domain, that is, on the circle or the line, see
{\it e.g.} \cite{ma, FCG, AV, Bultheel1999}.  This is why
the results by Khrushchev are of significance from the applied viewpoint
as well, which was one incentive for the authors to undertake 
the present study. This motivation is illustrated in the 
doctoral work by V.~Lunot \cite{Lunot}.

Unless the Schur function to be approximated possesses some
symmetry, though, there 
is no particular reason why Schur approximants should distinguish the origin.
It is thus natural to turn to multipoint Schur approximants, 
that play the role of Lagrange interpolating polynomials
in the present hyperbolic context, 
see definitions \eqref{zetan} and \eqref{e02} to come.
The role of orthogonal polynomials is then played by orthogonal
\emph{rational} functions with poles at the reflections of the interpolation 
points across the unit circle. Orthogonal rational functions, pioneered by
Dzrbasjan \cite{Dzrbaj}, were later
studied by Pan  \cite{Pan} and considerably expanded by Bultheel et al. 
\cite{Bultheel1999}, see also Langer-Lasarow \cite{Langer}. 
The last two references stress the connection with the multipoint 
Schur algorithm, and  the comprehensive exposition in \cite{Bultheel1999},
which contains further references, presents an account of Szeg\H o asymptotics
when the interpolation points are compactly supported in the disk.
 
The present article is concerned with the so-called determinate case 
(see condition \eqref{e04}) when the
interpolation sequence may have limit points on the circle, and its purpose
is two-fold. On the one hand, we derive analogues of Khrushchev's results
\cite{Khrushchev2001} on the convergence of Schur approximants in the 
multipoint case, and on the other hand we present a counterpart of the Szeg\H o
theory for the associated orthogonal rational functions. 
We limit ourselves to regular measures on the circle, whose density does not
vanish at limit points of the interpolation sequence, and we do not touch upon
what is perhaps the most important issue, namely how to choose the 
interpolation points in an optimal fashion as regards convergence rates. 
Nonetheless, the present paper seems first to propose asymptotics 
when the interpolation points approach the unit circle.  

Anyone writing on the subject faces the difficulty of expounding
the maze of formulas on which research can dwell. 
Our choice has been to
give a terse summary of what we use, along with references.

\subsection{Definitions}\label{nopre}
Let $\DD$ be the open unit disk and $\TT$ the unit circle.
A function $f$ is called {\it Schur}  if it belongs to the unit ball
of the Hardy space $H^\infty(\DD)$, {\it i.e.} if $f\in H^\infty(\DD)$
and $ ||f||_\infty\le 1$. The collection of Schur functions
is  called the {\it Schur class}, indicated by $\cS$.

The multipoint Schur algorithm goes as follows. 
Let $(\alpha_k)$, for $ k\in\NN$, be a  {\it fixed} 
sequence of points in 
$\mathbb{D}$. We set $\alpha_0=0$ by convention.
Define the elementary factor $\zeta_k$ by
\begin{equation}\label{zetan}
\zeta_k(z)=\frac{z-\alpha_k}{1-\bar\alpha_k z},  
\end{equation}
and put for $f\in\cS, k\ge 0$,
\begin{equation}\label{e02}
\left\{
\begin{array}{l}
f_0 = f, \\
\gamma_{k}= f_{k}(\alpha_{k+1}), \\
f_{k+1} = \dfrac{1}{\zeta_{k+1}} \dfrac{ f_{k} - \gamma_{k}}{1-\bar \gamma_{k} f_{k}}.
\end{array}\right. 
\end{equation}
 
We call $f_n$ the {\it Schur remainder} of $f$ of order $n$. The 
{\it Schur convergent}, or {\it Schur approximant} to $f$ of order $n$,
is defined from \eqref{e02} by formally computing $f$ in terms of $f_{n+1}$ and
$\gamma_k$ for $0\leq k\leq n$, and then substituting $f_{n+1}=0$ in the
resulting expression. 

It is a straightforward consequence of the maximum principle,
that the algorithm stops at some
finite $n$ ({\it i.e.} that $f_n$ is an unimodular constant) 
if and only if $f$ is a Blaschke product of degree $n$, 
namely a rational function in $\cS$ which is unimodular on $\TT$:
$$
B(z)=c\,\prod_{j=1}^n \frac{z-\beta_j}{1-\bar\beta_j z},  
$$
where $\beta_j\in\DD,\ |c|=1$.
{\it Throughout the paper, we assume that this is not the case}, 
so that the Schur algorithm, when applied to $f$ with some sequence 
$(\alpha_k)$,
produces an infinite sequence  
$(f_k)$. By the maximum principle, it is easily seen 
that $f_k$ is in turn Schur.
The complex numbers $\gamma_k$ appearing in the 
algorithm
are called the {\it Schur ({\rm or} Verblunsky) parameters} of $f$, and our assumption that $f$ is
not a finite Blaschke product is equivalent to the fact
that $\gamma_k\in\DD$ for all $k$.

The case where $\alpha_k\equiv0$, originally considered by
Schur \cite{Schur} and subsequently studied by many authors,
will be referred to as the \emph{classical} Schur algorithm. {\em Thus,
in the classical case, $\alpha_k=0$ and $\zeta_k(z)=z$ for all $k$, as opposed to
the multipoint version above where $(\alpha_k)$ may distribute 
arbitrarily in $\DD$}.

It is clear from \eqref{e02} (this is formalized in 
Proposition \ref{gamma_values_f}) that $\gamma_k$ is completely
determined by the interpolation values $f^{(j)}(\alpha_l)$ with
$0\leq j\leq n_l-1$, where $n_l$ is the multiplicity of $\alpha_l$ in the 
sequence $(\alpha_\ell)_{1\leq \ell\leq k+1}$ and 
the superscript $(j)$ indicates the $j$-th derivative.
In order for the Schur approximants to actually converge to $f$, it is 
thus necessary that
the sequence $(\alpha_k)$ be a uniqueness set in $H^\infty(\DD)$. 
This is equivalent to the negation of Blaschke condition:
\begin{equation}\label{e04}
\sum_k (1-|\alpha_k|)=+\infty.
\end{equation}
Of importance to us will be the equivalence of (\ref{e04}) with the 
density of rational functions having poles at the points 
$(1/\overline\alpha_k)$ in every Hardy space $H^p(\DD)$, $1\leq p<\infty$, 
as well as in the disk algebra $A(\mathbb{D})$
\cite[App. A]{Akhiezer}. 
 
Next, we recall a basic construction relating the classical Schur algorithm
to orthogonal polynomials on 
$\TT$, see {\it e.g.} \cite{Khrushchev2001, Simon1}. 
For $\mu$ a Borel probability measure on $\TT$,  we let
$\mu_{ac}$ and $\mu_s$  respectively be its  absolutely continuous and 
singular components with respect to $m$,  
the Lebesgue measure given by 
$dm(t)=dt/(2\pi i t)=\frac 1{2\pi} d\theta$ where
$t=e^{i\theta}\in\TT$. We further put $\mu'=d\mu_{ac}/dm$ 
so that $d\mu=\mu' dm+d\mu_s$. 

To $f\in\cS$, we associate two probability measures $\mu,\tilde\mu$ 
on $\TT$ by the relations
\begin{equation}\label{e06}
F_\mu(z)=\frac{1+zf}{1-zf}=\int_\TT \frac{t+z}{t-z} d\mu(t),\quad
F_{\ti\mu}(z)=\frac{1-zf}{1+zf}=\int_\TT \frac{t+z}{t-z} d\tilde\mu(t).
\end{equation}
Clearly $F_\mu$ is a Carath\'eodory  function, 
{\it i.e. } $\mathrm{Re}\, F(z)>0,\ z\in\DD$; moreover $F(0)=1$.
We call $F_\mu$ the \emph{Herglotz transform} of $\mu$,
and the representation \eqref{e06} is possible because
every Carath\'eodory function is uniquely the Herglotz transform of a 
finite positive 
measure. 
From the Fatou theorems \cite[Ch. I, Sect. D]{Koosis}, we note that
\begin{equation}
\label{fFF}
\mu^\prime=\mathrm{Re}\, F_\mu=\frac{1-|f|^2}{|1-zf|^2},\qquad
\ \ 
\lim_{r\to1} \mathrm{Re}\, F_\mu(re^{i\theta})=+\infty,
\end{equation}
$m$-a.e. and $\mu_s$-a.e., respectively. Similar considerations hold for $\tilde\mu$.

Let $(\phi_n)$ and $(\psi_n)$ be the orthonormal 
polynomials with respect to $\mu$ and $\tilde\mu$:
\begin{equation}\label{e05}
\int_\TT \phi_n\ovl\phi_m d\mu=\delta_{nm},\qquad \int_\TT \psi_n\ovl\psi_m 
d\tilde\mu=\delta_{nm},
\end{equation} 
here $\delta_{nm}$ is the Kronecker symbol.  Our assumption that $f$ is 
not a finite Blaschke product means that $\mu$ and $\tilde\mu$ have infinite
support, therefore $\phi_n$, $\psi_n$ have exact degree $n$.  
The sequences $(\phi_n)$ and $(\psi_n)$ are called respectively
the orthonormal polynomials of first and second kind 
associated with $\mu$.
Clearly $\phi_n$ and $\psi_n$ are unique up to a multiplicative
unimodular constant. We normalize them so that  their respective 
leading coefficients  $k_n$ and $k_n^\prime$ are positive. 

For a polynomial $\phi$ of degree $n$, put
$\phi^*(z)=z^n \overline{\phi(1/\bar z)}$. This is again
a polynomial of degree $n$. Note that $k_n=\ovl{\phi_n^*(0)}$.
The coefficients 
\begin{equation}\label{e041}
\ti\gamma_n=\ti\gamma_n(\mu)=-\frac{\ovl{\phi_{n+1}(0)}}{k_{n+1}}
\end{equation}
are called {\it the Geronimus parameters} associated with $(\phi_n)$ 
(or with $\mu$). 
 
The following remarkable theorem, named after Geronimus, was proven almost 
simultaneously by Geronimus \cite{Geronimus} and Wall \cite{wa1}.
\begin{theoremm} Let $f\in \cS$. If $\alpha_k\equiv0$,
the Schur parameters and the Geronimus  parameters 
coincide, {\it i.e.} $\gamma_n=\ti\gamma_n,\ n\ge 0$.
\end{theoremm}

Since trading $\mu$ for $\ti\mu$ is tantamount to change $f$ into $-f$, 
a corollary is that
$\ti\gamma(\mu)=-\ti\gamma(\ti\mu)$.

We turn to the multipoint version of Geronimus' theorem,
which is due essentially to Bultheel et al. \cite{Bultheel1999} 
although the first
explicit statement is apparently in Langer-Lasarow \cite{Langer}. 
For this, orthogonal polynomials need to be generalized into
orthogonal rational functions whose construction we now explain.
Define  the ``partial'' Blaschke products $\mathcal{B}_k$ by
\begin{equation}\label{Bn}
\mathcal{B}_0(z)  =  1,\quad  \mathcal{B}_k(z)  =  
\mathcal{B}_{k-1}(z)\zeta_k(z),  
\end{equation}
where $\zeta_k$ is given by \eqref{zetan} and $k\geq 1$.
The functions $\{\mathcal{B}_0,\mathcal{B}_1,\ldots,\mathcal{B}_n\}$ span the space
\begin{equation}
        \cL_n=\left\{\frac{p_n}{\pi_n}~:~\pi_n(z)=\prod_{k=1}^n(1-\bar\alpha_k z),~~p_n\in\cP_n\right\},
\label{Ln}
\end{equation}
where $\cP_n$ stands for the space of algebraic polynomials of degree 
at most $n$. In the classical case, that is when $\alpha_k=0$ for all $k$,
$\cL_n$ coincides with $\cP_n$.

Given a function $g$, we introduce the parahermitian conjugate $g_*$ defined by $g_*(z)=\overline {g(1/\bar z)}$. Observe that $|g_*|=|g|$ on $\TT$ and that
${\zeta_n}_{*} = {\zeta_n}^{-1}$, ${\mathcal{B}_k}_* = {\mathcal{B}_k}^{-1}$.
For $g\in \cL_n$, we set $g^*=\mathcal{B}_nf_*$; clearly, $g^*\in\cL_n$. 
There is no 
notational discrepancy since in the classical case the star operation 
agrees with the definition we gave before.
Put $\mathcal{B}_{n,i}=\prod_{k=i}^{n}\zeta_k$. 
Each $g\in\cL_n$ can be uniquely decomposed in the form
$$
g = a_n \mathcal{B}_n + a_{n-1} \mathcal{B}_{n-1} + \dots + a_1 \mathcal{B}_1 + a_0,
$$
and then
$$
g^* = \bar a_0 \mathcal{B}_{n,1} + \bar a_{1} \mathcal{B}_{n,2} + \dots + \bar a_{n-2} \mathcal{B}_{n,n-1} + \bar a_{n-1} \mathcal{B}_{n,n}  + \bar a_n. 
$$
It is plain that 
$a_n=\overline{g^*(\alpha_n)}$ and $a_0=g(\alpha_1)$.

Now, pick  a Schur function $f$ which is not a Blaschke product,
denote its Herglotz measure by $\mu$ 
\eqref{e06}, and consider $\cL_n$ as a subspace of $L^2(\mu)$. This is 
possible  since $\mu$ has infinite support.
Let $(\phi_k)_{0\le k\le n}$ be an orthonormal basis for $\cL_n$ such that $\phi_0=1$ and $\phi_k\in \cL_k\setminus\cL_{k-1}$. Such a basis is easily 
obtained on applying the Gram-Schmidt orthonormalization process to 
$\mathcal{B}_0, \mathcal{B}_1, \dots, \mathcal{B}_n$. 
We customary write 
\begin{equation}
        \label{phin}
        \phi_n=\kappa_n \mathcal{B}_n+a_{n,n-1}\mathcal{B}_{n-1}+\ldots+a_{n,1}\mathcal{B}_1+a_{n,0}\mathcal{B}_0,  
\end{equation}
where $\kappa_n = \overline{\phi_n^*(\alpha_n)}$.
\begin{definition}\label{d01} The functions $(\phi_k)$ are called 
the orthogonal rational functions of the first kind 
associated to $(\alpha_k)$ and  $\mu$.
\end{definition}
The $(\psi_n)$ arising from embedding $\cL_n$ to $L^2(\ti\mu)$
are called the orthogonal rational functions of the second kind.
Clearly, the orthogonal rational functions $(\phi_n), (\psi_n)$
defined in \eqref{phin} reduce to the orthonormal polynomials 
from \eqref{e05} in the classical case.

Generically, the dependence on the nodes $(\alpha_k)$ and the measure 
$\mu$ will be omitted. The words ``orthogonal rational function'' will be 
abbreviated as ORF or OR-function.

The definition of the {\it Geronimus parameters} $(\ti\gamma_k)$ for
OR-functions is 
\[\tilde{\gamma}_n=-\frac{\overline{\phi_{n}(\alpha_{n-1})}}{\overline{\phi_{n}^*(\alpha_{n-1})}}, \quad n\geq1.\]
Note we do not define $\ti\gamma_0$ and there is a shift of index
as compared to \eqref{e041}.

It is quite nontrivial that one can relate the Schur algorithm 
\eqref{e02} and the ORFs \eqref{phin} in the multipoint case as well: 
\begin{theoremm}[\cite{Bultheel1999, Langer}] Let $(\alpha_k),\ f\in\cS,$ and  the ORFs $(\phi_n)$ be as above. Then the multipoint Schur and Geronimus parameters coincide, {\it i.e. } $\gamma_k=\ti\gamma_{k+1}$.
\end{theoremm}
We prove this fundamental result in 
Section \ref{s3} 
for the sake of completeness.

\subsection{Discussion of the main results}
\label{mainresults}
The convergence properties of the Schur approximants and of the
ORFs $(\phi_n)$ are the main address of the present work, which is in part inspired by the results obtained 
by Khrushchev \cite{Khrushchev2001}.
To better see the parallel between the classical and the multipoint case, 
we give below a sample of results  from \cite{Khrushchev2001}
in the classical situation, and have them followed by their multipoint
counterparts,
numbered with a prime superscript; we connect these counterparts to the 
forthcoming results
in between parentheses.

 We say that a measure $\mu$ is Erd\H os, iff $\mu'>0$ a.e. on $\TT$. 
This is equivalent to say that $|f|<1$ a.e. on $\TT$.
\begin{theorem1}[{\cite{Khrushchev2001}, Theorem 1}]\label{tt1} Let $f\in\cS$  and $\mu$ be its Herglotz measure. If $\alpha_k\equiv0$, then
$\mu$ is Erd\H os if and only if the Schur remainders $f_n$ satisfy
$$
\lim_n \int_\TT |f_n|^2 dm=0.
$$
\end{theorem1}

The next result is stated in terms of 
the classical Wall polynomials $A_n$, $B_n$ of $f$ 
\cite[Sect. 4]{Khrushchev2001}, obtained from 
Definition \ref{d1} below by setting $\alpha_k\equiv 0$. By
definition of the Wall polynomials, the ratio $A_n/B_n$ is the {\it Schur approximant} to $f$ of 
degree $n$. 
Recall that the pseudohyperbolic distance on $\DD$ is defined as $\rho(z,w)=|z-w|/|1-\bar w z|,\ z,w\in \DD$.   
 
\begin{theorem2}[{\cite{Khrushchev2001}, Corollary 2.4}]\label{tt2}
A measure $\mu$ is Erd\H os if and only if
$$
\lim_n \int_\TT\rho\left(f,\frac{A_n}{B_n}\right)^2 dm=0.
$$
\end{theorem2}

We shall see that, in the multipoint situation  when the sequence 
$(\alpha_k)$ accumulates on the unit circle, 
the conclusions of Theorems 1 and 2 get 
localized around the accumulation points of $(\alpha_k)$
on $\TT$ so that $L^2$-norms
get weighted by the Poisson kernel at $\alpha_{n+1}$. 
This is why, somewhat reminiscently of the Fatou theorem,
we put extra-conditions on $\mu$, locally
around such points, to derive convergence properties.
Namely, let $Acc(\alpha_k)=\ovl{(\alpha_k)}\bsl (\alpha_k)$ be
the set of accumulation points of $(\alpha_k)$; the bar (or $clos\, (.) $) 
stands for the closure of a set.
The following assumptions play an important role in our proofs
\begin{eqnarray}\label{e08}
&&\mu'\in {\mathcal C}(\cO(Acc(\alpha_k)\cap\TT)),\\
&& \mu'>0\ \ \mathrm{on}\ \ \cO(Acc(\alpha_k)\cap\TT), \label{e082}\\
&&  \{Acc(\alpha_k)\cap\TT\}\subset \TT\bsl \supp\mu_s, \label{e083}
\end{eqnarray}
where, for $A\subset\TT$, $\cO(A)$ designates an open neighborhood of $A$ 
in $\TT$
and $\mathcal{C}(A)$ is the space of continuous functions on  $A$.  
The {\it closed} support of $\mu_s$ is denoted by $\supp\mu_s$. 
When the sequence $(\alpha_k)$ accumulates nontangentially on
$Acc(\alpha_k)\cap\TT$, meaning that every convergent subsequence to
$\xi\in\TT$  tends to the latter
nontangentially, two weaker  substitutes for \eqref{e08}, \eqref{e082} are
also of interest:
\begin{eqnarray}
&&\mathrm{ each}\ \xi\in Acc(\alpha_k)\cap\TT\ \mathrm{is\ a\  Lebesgue \ point \ of\ } 
\mu',\ \sqrt{\mu'}, \label{e08p}\\ 
&& \mathrm{and\ }\mu'(\xi)>0; \nonumber\\
&& \mu'\mathrm{\ is\  upper\ semicontinuous,\ }0<\delta<\mu'<M<\infty \mathrm{\ on} \label{e082p}\\
&& \mathcal{O}(Acc(\alpha_k)\cap\TT),  \mathrm{\ \ and\ \ each \ }\ \xi\in Acc(\alpha_k)\cap\TT
\mathrm{\ \ is\ a} \nonumber\\
&& \mathrm{\  Lebesgue \ point \ of}\ \log \mu'.\nonumber
\end{eqnarray}
From \eqref{fFF}, we see that \eqref{e08} and  \eqref{e082}
may be ascertained in terms of $f$,  namely 
$f\in {\mathcal C}({\mathcal O}(Acc\,(\alpha_k)\cap\TT))$ and $|f|<1$ 
there, while $(Acc (\alpha_k)\cap\TT)\subset \TT\bsl clos\, \{z:\,zf(z)=1\}$.  

The multipoint analogues of the previous theorems go as follows.
\begin{theorem1'}[{Corollary \ref{c03} and Theorem \ref{convpAB}}]
Let  \eqref{e04}, \eqref{e08}-\eqref{e083} hold, and  $|f| < 1$ a.e. on $\TT$. Then
\[
\lim_{k} \int{|f_{k}|^2 P(.,\alpha_k) dm}=0.
\]
If $(\alpha_k)$ accumulates nontangentially on
$Acc(\alpha_k)\cap\TT$, then one can replace hypotheses
\eqref{e08} and \eqref{e082} with \eqref{e08p}.
\end{theorem1'}

Above, $P(.,\alpha_k)$ is the Poisson kernel at $\alpha_k$
on $\DD$ \eqref{e09}.
Denote by $W^{q,p}(\TT)$ the Sobolev spaces on $\TT$ (see Section \ref{ss03} for more details). Recall that  $(A_n), (B_n)$  are the Wall rational functions from Definition \ref{d1} corresponding to $f\in \cS$.

\begin{theorem1''}
Let $\mu$ be absolutely continuous 
with $\mu'\in W^{1-1/p,p}(\TT)$ for some $p>4$, and 
$\mu'>0$ on some neighborhood $\cO(Acc(\alpha_k)\cap\TT)$. 
Then 
\[
\lim_n\left\|\left(f-\frac{A_n}{B_n}\right) \sqrt{P(.,\alpha_{n+1})}\right\|_\infty=0.
\]
\end{theorem1''}

Remarks on the converse to Theorem 1' follow Theorems \ref{t001}, \ref{t002}.
 
\begin{theorem2'}[{Theorem \ref{t01}}]
Assumptions being as in Theorem 1', we have
\[
\lim_{k} \int_\TT \rho\left(f,\frac{A_k}{B_k}\right)^2 P(.,\alpha_{k+1}) dm=0.
\]
\end{theorem2'}
The final point of the paper is to carry over the
Szeg\H o  theory to the multipoint setting. Recall that
a measure $\mu$ is called {\it Szeg\H o}  (notation: $\mu\in\sz$) iff 
$\log\mu'\in L^1(\TT)$.
For $\mu\in\sz$, the associated Szeg\H{o} function $S$ is   
\begin{equation}\label{e081}
S(z) = S[\mu](z):=\exp \left(\frac{1}{2} \int_{\mathbb{T}} \frac{t+z}{t-z} 
\log\mu' dm(t) \right).
\end{equation}
The function $S$ given by  \eqref{e081} is the so-called {\it outer} function in 
$H^2(\DD)$ such that
$|S|^2=\mu'$ a.e. on $\TT$, normalized so that $S(0)>0$.  

The first version of the next theorem, which addresses the classical
case, was proven by Szeg\H o \cite{Szego}. Subsequent improvement were obtained by Geronimus \cite{Geronimus1}, Krein \cite{kr1}, and others; 
see Simon \cite{Simon1} for the discussion and a full list of references.  
Some of the latest improvements are due to Nikishin-Sorokin \cite{Nikishin}, Peherstorfer-Yuditskii \cite{yu1}. A generalized version of Szeg\H o condition is treated in Denisov-Kupin \cite{dk1, dk2}.
\begin{theorem3} Let $\mu\in\sz$ and $(\phi_n)$ be the
corresponding orthonormal polynomials \eqref{e05}. Then
\begin{itemize}
\item $\lim_n (S\phi_n^*)(0)= 1$; more generally,  $\lim_n (S\phi_n^*)(z)= 1$ for $z\in\DD$. 
\item  
$\dsp \lim_n \int_\TT |S\phi^*_n-1|^2 dm=0.$
\end{itemize} 
Moreover, $\mu\in\sz$  if and only if
$$
\lim_n \int_\TT\mathfrak{P}\left(f,\frac{A_n}{B_n}\right)^2 dm=0,
$$
where $\mathfrak{P}(.,.)$  is the hyperbolic distance on $\DD$
\eqref{e07}. Equivalently,
$$
\lim_n \int_\TT \log(1-|f_n|^2)\, dm=0.
$$
\end{theorem3}
The last assertion of the theorem concerning the hyperbolic distance
is from Khrushchev \cite{Khrushchev2001}, Theorem 2.6.

A multipoint analogue to the previous theorem
when $(\alpha_n)$ is compactly supported in $\DD$
is Theorem 9.6.9 from Bultheel et al. \cite{Bultheel1999}; its generalization to sequences $(\alpha_k)$
meeting \eqref{e04} is given below.
It is more difficult
and requires some preparation. 
It relies on {\it a priori} pointwise estimates of $(\phi_n)$ 
(see Proposition \ref{bornephin}), that play here the role
of classical bounds by
Szeg\H o and Geronimus  \cite[Ch. 12]{Szego}, 
\cite[Ch. 4]{Geronimus1}. Such estimates are new even in the polynomial case,
as they handle some situations where $\mu'$ may vanish. 
Their proof in turn depends 
on $\ovl\partial$-estimates and Sobolev embeddings.
As compared to the case where $(\alpha_n)$ is compactly supported in $\DD$, 
the result below is of new type in that $S\phi_n^*$ is 
asymptotic to a normalized Cauchy kernel at the last interpolation point,
which is unbounded
when $(\alpha_n)$ approaches $\TT$. 
Here is a combination of Theorem \ref{cvg hyperbolic}, 
Theorem \ref{c04} and Corollary \ref{c01} to come:
\begin{theorem3'}\label{t3'}
Let \eqref{e04}, \eqref{e08}-\eqref{e083}  be in force,
with  $\mu\in\sz$. 
Then
\begin{itemize}
\item $\lim_{n} |\phi_n^*(\alpha_n)|^2 |S(\alpha_n)|^2 (1-|\alpha_n|^2) = 1$;
more generally, for any sequence $(z_n)\subset \DD$, it holds
\[
\lim_{n}\left\{ \phi_n^*(z_n)S(z_n) \sqrt{1-|z_n|^2} 
-\beta_n \frac{\sqrt{1-|\alpha_n|^2}\sqrt{1-|z_n|^2}}{1-\overline{\alpha}_nz_n}
\right\}=0,
\]
where $\beta_n=(S\phi^*_n)(\alpha_n)/|(S\phi^*_n)(\alpha_n)|$.
In particular, for a fixed $z\in\DD$,
$$
\lim_n \left\{S\phi^*_n(z) -\beta_n\frac{\sqrt{1-|\alpha_n|^2}}{1-\overline\alpha_n z}\right\}=0.
$$

\item We also have
\[
\lim_n \left\| S\phi^*_n(z) -\beta_n\frac{\sqrt{1-|\alpha_n|^2}}{1-\overline\alpha_n z}\right\|_2=0.
\]
\item Moreover
$$
\lim_{n} \int_\TT \mathfrak{P}\left(f,\frac{A_n}{B_n}\right)^2 P(.,\alpha_{n+1}) dm=0,
$$
in particular
$$
\lim_n \int_\TT\log(1-|f_n|^2) P(.,\alpha_{n}) dm=0.
$$
\end{itemize}
If $(\alpha_k)$ accumulates nontangentially on
$Acc(\alpha_k)\cap\TT$, then one can replace
\eqref{e08} and \eqref{e082} by \eqref{e082p}.
\end{theorem3'}
It would be interesting to know how much these assumptions can be 
relaxed. In particular, we shall give an example where the conclusion of
Theorem 3' holds although \eqref{e082} fails.

The paper is organized as follows.  The multipoint Schur algorithm, 
its connections to continued fractions, the Wall rational functions and 
the Schur parameters are discussed in Section \ref{s2}. 
Section \ref{s3} introduces the ORFs $(\phi_n), (\psi_n),$ and expresses them through Geronimus parameters and transfer matrices. The construction is then used to prove Geronimus' theorem and its corollaries.  Although we use a different normalization, by and large, 
the content of Section \ref{s3} is borrowed from Bultheel 
et al. \cite{Bultheel1999}.  The convergence of Schur remainders and Wall RFs is studied in Sections 
\ref{s5} and \ref{s6}. Section \ref{s7} is devoted to the discussion of 
the Szeg\H o-type theorem and its corollaries. 

\subsection{Some notation}\label{ss03}
As already mentioned,  the closure of $A\subset \CC$ is indicated by $clos\, A$ or $\ovl A$,
while $\cO(A)$ designates an open neighborhood of $A$ in $\TT$.
The normalized Lebesgue measure on $\TT$ is denoted by 
$m$, and the measure of $A\subset\TT$ is denoted by 
$|A|$. We put ${\mathcal C}(A)$ for the space of continuous functions on 
$A$. The symbol 
$||.||_p$ stands for the usual norm on the Lebesgue space
$L^p(\TT),\ 1\le p\le\infty$; 
\emph{when $p=2$, the subindex is usually dropped}.
The classical analytic Hardy spaces of the disk are denoted by 
$H^p(\DD), 1\le p\le\infty$, and 
$A(\DD)$ is the disk algebra,  comprised of analytic
functions in $\DD$ 
that extend continuously to $\overline{\DD}$, endowed with the $sup$ norm.
Standard references on the subject are the books by Duren \cite{Duren}, 
Garnett \cite{Garnett}, Koosis \cite{Koosis}, from which we often 
quote basic facts without further citation. 
In particular, $H^p$-functions have well-defined nontangential limits 
in $L^p(\TT)$, and we use 
the same notation for the function in $\DD$ and its trace on $\TT$.

Every real-valued $\varphi\in L^1(\TT)$ is $m$-a.e. the real part
of the nontangential limit of the complex analytic function
\begin{equation}\label{defHT}
F_\varphi(z)=\int_\TT \frac{t+z}{t-z} \varphi(t)\,dt,\ \ \ \ z\in\DD,
\end{equation}
which is called the \emph{Herglotz transform} of $\varphi$.
The map sending $\varphi$ to the imaginary part of $F_\varphi$ is the
\emph{conjugation operator}, denoted with the superscript ``$\,\check{}\,$'',
{\it i.e.} $F_\varphi=\varphi+i\check{\varphi}$;
it extends linearly to complex-valued functions.
By a theorem of M. Riesz, the conjugation operator
acts on $L^p(\TT), 1<p<\infty$.  Moreover, since
it is of convolution type, it commutes with $d/|dt|$
and integrating by parts one sees that it 
also acts on $W^{1,p}(\TT)$, the space of absolutely continuous functions
with $L^p$ derivative on $\TT$.

The Poisson kernel on $\DD$ is 
\begin{equation}\label{e09}
P(z,w)=P_w(z)=(1-|w|^2)/|z-w|^2,
\end{equation}
where  $z\in\TT,  w\in\DD$.

We shall need some basic facts from Sobolev space theory for which we refer
the reader to Adams-Fournier \cite{Adams}.
In particular, for $I\subset\TT$ an open arc,
those $\varphi$ for which
\begin{equation}
\label{defW1-1/p}
\int_{t,t'\in I}
\left|\frac{\varphi(t)-\varphi(t')}{t-t'}\right|^p\,dm(t)dm(t')<\infty,
\end{equation}
with $1<p<\infty$, form the \emph{fractional Sobolev space} 
$W^{1-1/p,p}(I)$ (that coincides with the Besov space $B^{1-1/p,p}_p(I)$).
It is a real interpolation space between $W^{1,p}(I)$ and $L^p(I)$:
$W^{1-1/p,p}(I)=[W^{1,p}(I),L^p(I)]_{1/p}$, that
embeds compactly into  $L^p(I)$. By interpolation
\cite[Sect.~7.3.2, Theorem 7.3.2]{Adams}, the conjugation 
operator also acts on $W^{1-1/p,p}(\TT)$.

When $h$ is defined on $E\subset\CC$ and $0<\alpha<1$,
we say that $h$ is H\"older continuous 
of exponent $\alpha$
if there is a constant $C > 0$ such that 
$|h(t)-h(t')|\le C|t-t'|^\alpha$ 
for all $t,t'\in E$. We then write $h\in H_\alpha(E)$.
Note that $H_\alpha(I)\subset W^{1-1/p,p}(I)$ if $1<p<1/(1-\alpha)$.
By the Sobolev embedding theorem \cite[Theorem 4.12]{Adams},  
it holds conversely that $W^{1-1/p,p}(I)\subset H_{1-2/p}(I)$ for $p>2$.

\section{Wall rational functions}\label{s2}
In this section we rewrite Section 4 from
\cite{Khrushchev2001} for the multipoint case. The presentation is very close to the original, and only some technical details are different. Roughly speaking, one mainly has to replace $z^n$ with $\cB_n$; that is why we generically give results accompanied by precise references to \cite{Khrushchev2001} and omit the proofs.
We start recalling basic definitions on continued fractions \cite{Wall}. 

A continued fraction is an infinite expression of the form 
\[
 b_0 + \frac{a_1}{b_1+\frac{a_2}{b_2+\frac{a_3}{b_3+\frac{a_4}{\dots}}}}.
\]
We conform the more economic notation
\[
  b_0 + \frac{a_1}{b_1}
\begin{array}{c}
        \\
        +
\end{array}
\frac{a_2}{b_2}
\begin{array}{c}
        \\
        +
\end{array}
\frac{a_3}{b_3}
\begin{array}{c}
        \\
        + \dots .
\end{array}
\]
For any complex-valued $\omega$, we let $t_0(\omega)=b_0 + \omega$ and, 
for $k\ge 1$,
\[ 
t_k(\omega)=\frac{a_k}{b_k+\omega}.
\]
By definition, the $n$-th convergent ${P_n}/{Q_n}$ of the continued fraction is 
 \[
 \frac{P_n}{Q_n}=t_0 \circ t_1 \circ \dots \circ t_n(0)
= b_0 + \frac{a_1}{b_1}
\begin{array}{c} \\+ \end{array}
\frac{a_2}{b_2}
\begin{array}{c} \\+ \end{array}
\ldots
\begin{array}{c} \\+ \end{array}
\frac{a_n}{b_n}.
\]
 
\begin{proposition}[{\cite{Khrushchev2001}, relations (3.2)-(3.4)}]
The quantities $P_n$ and $Q_n$ can be computed according
 to the recurrence relations:
\[
\left\{
\begin{array}{l}
P_{-1}=1, Q_{-1}=0,\\
P_0=b_0, Q_0=1,\\
P_{k+1}=b_{k+1} P_k + a_{k+1} P_{k-1}\\
Q_{k+1}=b_{k+1} Q_k + a_{k+1} Q_{k-1} 
\end{array}
\right.
\]
for $k\ge 0$. More generally,
\[
\frac{P_{n-1} \omega + P_n}{Q_{n-1} \omega + Q_n}= t_0 \circ t_1 \circ \dots \circ t_n(\omega).
\]
\label{Euler_recu}
\end{proposition}
First we record the following fact.
\begin{proposition}
For $k\ge 1$, $\gamma_k$ depends only on $f^{(i)}(\alpha_j)$, $1\le j\le k+1$, $0 \le i < m_j$,  where $m_j$ is the multiplicity of $\alpha_j$ at the 
$k$-th step, 
{\it i.e. } $m_j$ is the number of times the value of $\alpha_j$ 
enters $(\alpha_l)_{1\le l\le k+1}$.
\label{gamma_values_f}
\end{proposition}
\begin{proof}
Noticing in case of repetitions that $f_j(\alpha_j)=f_{j-1}^{\prime}(\alpha_j)  \frac{1-|\alpha_j|^2}{1-|f_{j-1}(\alpha_j)|^2}$, the proof is immediate by induction on \eqref{e02}.
\end{proof}

We now rewrite the recursive step of \eqref{e02} as
\begin{equation} \label{SchurStep}
f_{k-1} = \gamma_{k-1} + \frac{(1-|\gamma_{k-1}|^2)  \zeta_k}{\bar \gamma_{k-1} \zeta_k + \frac{1}{f_k}}.
\end{equation}
For $\omega\in\DD\setminus\{0\}$, set   
\begin{equation}\label{tau_k}
\tau_k(\omega)=\tau_k(\omega,z):=\gamma_{k} + \frac{(1-|\gamma_{k}|^2)  \zeta_{k+1}}{\bar \gamma_{k} \zeta_{k+1} + \frac{1}{\omega}},
\end{equation}
and put $\tau_k(0)=\gamma_k$. Hence, $f_k=\tau_k(f_{k+1})$ and
\begin{equation}
 f  =  \tau_0 \circ \tau_1 \circ \dots \circ \tau_n (f_{n+1}).
\label{f tau}
\end{equation}
In a way reminiscent of how we defined $P_n/Q_n$,
we obtain the Schur convergent $R_n$ of degree $n$ upon replacing   
$f_{n+1}$ by $0$ in \eqref{f tau}, that is,
\begin{equation}
R_n= \tau_0\circ\tau_1\circ\dots\circ\tau_{n-1}\circ\tau_{n}(0)
=  \tau_0\circ\tau_1\circ\dots\circ\tau_{n-1}(\gamma_n).
\label{Rn}
\end{equation}

\begin{proposition}
The rational function $R_n$ interpolates $f$ at 
$(\alpha_k)_{1 \le k \le n+1}$, counting multiplicities,
and  their first  $n+1$ Schur parameters coincide.
\end{proposition}

\begin{proof}
Note that $\tau_k(\omega,\alpha_{k+1})=\gamma_k$ is independent of $\omega$.  
Thus, for $0 \le k \le n$,
\begin{eqnarray*}
f(\alpha_{k+1})&=& \tau_0 \circ  \dots \circ \tau_k (\tau_{k+1} \circ \dots \circ \tau_n(f_{n+1}), \alpha_{k+1})\\
& =& \tau_0 \circ  \dots \circ \tau_k (\tau_{k+1} \circ \dots \circ 
\tau_n (0), \alpha_{k+1})\\
& =& R_n(\alpha_{k+1}).
\end{eqnarray*}
Consequently, $R_n$ interpolates $f$ at the point $\alpha_{k+1}$.

The remaining part of the claim is proven by induction. The base of induction being obvious, suppose that the  $k$ first Schur parameters of $f$ and
$R_n$ coincide.
Then, denoting $R_n^{[1]}, \dots R_n^{[n]}$ the Schur remainders
of $R_n$, we see that $R_n^{[k]}=\tau_{k-1}^{-1} \circ  \dots \circ \tau_0^{-1}(R_n)$, and
\begin{eqnarray*}
R_n^{[k]}(\alpha_{k+1}) &=& \tau_{k-1}^{-1} \circ  \dots \circ \tau_0^{-1}(R_n, \alpha_{k+1})\\
&=& \tau_{k-1}^{-1} \circ  \dots \circ \tau_0^{-1} \circ \tau_0\circ\tau_1\circ\dots\circ\tau_{n-1}(\gamma_n, \alpha_{k+1})\\
&=& \tau_k(\tau_{k+1}\circ\dots\circ\tau_{n-1}(\gamma_n),\alpha_{k+1}) = \gamma_{k+1}.
\end{eqnarray*}
Therefore, the $k+1$-th Schur parameter of $R_n$ and $f$ coincide.
\end{proof}
The Schur algorithm can be readily connected to the continued fractions. 
Indeed, let  $P_n/Q_n$ be the sequence of convergents associated to  
\begin{equation}
  \gamma_0 + \frac{(1-|\gamma_0|^2) \zeta_1}{\bar \gamma_0 \zeta_1}
\begin{array}{c}
        \\
        +
\end{array}
\frac{1}{\gamma_1}
\begin{array}{c}
        \\
        +
\end{array}
\frac{(1-|\gamma_1|^2) \zeta_2}{\bar \gamma_1 \zeta_2}
\begin{array}{c}
        \\
        + \dots.
\end{array}
\label{fraction continue algo}
\end{equation}
Then, the functions $R_n$ are none but the $P_{2n}/Q_{2n}$.

For $n\ge 1$, we have by Proposition \ref{Euler_recu}
\begin{equation}
\begin{split}
        P_{2n} & =  \gamma_{n} P_{2n-1} + P_{2n-2}\\    
        Q_{2n} & =  \gamma_{n} Q_{2n-1} + Q_{2n-2} \\
        P_{2n-1} & =  \bar \gamma_{n-1} \zeta_{n} P_{2n-2} + (1-|\gamma_{n-1}|^2) \zeta_{n} P_{2n-3}\\
        Q_{2n-1} & =  \bar \gamma_{n-1} \zeta_{n} Q_{2n-2} + (1-|\gamma_{n-1}|^2) \zeta_{n} Q_{2n-3}
\end{split}
\label{Euler}
\end{equation}
with 
\[
\begin{array}{llll}
 P_{-1}=1, & P_0 = \gamma_0, & Q_{-1}=0, & Q_0=1.
\end{array}
\]
For $P_{2n}$ and $Q_{2n}$, we easily prove the next lemma.
 
\begin{lemma}[{\cite{Khrushchev2001}, Lemma 4.1}]
\label{relations_order_even_and_odd}
 For $n \ge 0$, we have $P_{2n+1}, Q_{2n+1} \in \mathcal{L}_{n+1}$, $P_{2n}, Q_{2n} \in \mathcal{L}_{n}$ and
\[
 \begin{array}{ll}
  P_{2n+1}=\zeta_{n+1} Q_{2n}^*, & Q_{2n+1}= \zeta_{n+1} P_{2n}^*
 \end{array}.
\]
\end{lemma}

Mimicking \cite{Khrushchev2001}, formulas (4.5), (4.12), we get 
\begin{equation}
\label{recu_Wall}
\left[
        \begin{array}{ll}
        Q_{2n}^* & P_{2n}^*\\
        P_{2n} & Q_{2n}
        \end{array}
\right]
=
\left(
\prod_{k=n}^{1}
\left[
        \begin{array}{ll}
        1 & \bar \gamma_k\\
        \gamma_k & 1
        \end{array}
\right]
\left[
        \begin{array}{ll}
        \zeta_{k} & 0\\
        0 & 1
        \end{array}
\right]
\right)
\left[
        \begin{array}{ll}
        1 & \bar{\gamma_0}\\
        \gamma_0 & 1
        \end{array}
\right],
\end{equation}
with $n \ge 1$.
Let us set $A_n= P_{2n}$,  
$B_n = Q_{2n}$ and choose  the representative $R_n=A_n/B_n$ for $R_n$.

\begin{definition}\label{d1}
${A_n}$ and ${B_n}$ are called the $n$-th {\emph Wall rational functions} 
associated to the Schur function $f$ and the sequence $(\alpha_k)$.
\end{definition}

In the new notation, the previous relation reads as
\begin{proposition}[{\cite{Khrushchev2001}, relation (4.12)}]
We have 
\begin{eqnarray}
\left[
        \begin{array}{ll}
        B_{n}^* & A_{n}^*\\
        A_{n} & B_{n}
        \end{array}
\right]
& = &
\left(
\prod_{k=n}^{1}
\left[
        \begin{array}{ll}
        1 & \bar \gamma_k\\
        \gamma_k & 1
        \end{array}
\right]
\left[
        \begin{array}{ll}
        \zeta_{k} & 0\\
        0 & 1
        \end{array}
\right]
\right)
\left[
        \begin{array}{ll}
        1 & \bar{\gamma_0}\\
        \gamma_0 & 1
        \end{array}
\right].
\label{recuPQ}
\end{eqnarray}
\label{recurrence_An_Bn}
\end{proposition}

The dependence of $A_n$, $B_n$ on $f$ and $(\alpha_k)$ will be usually 
dropped. For convenience, we abbreviate ``Wall rational function'' 
as WRF or Wall RF.

\begin{corollary}[{\cite{Khrushchev2001}, relations (4.14), (4.15)}]
\label{det_A_B_on_T}
The Wall RFs $A_n, B_n$ have the following properties:
\begin{enumerate}
\item $B_n(z)B_n^*(z) - A_n(z) A_n^*(z) = \mathcal{B}_{n}(z) \omega_n$,
 \item $ |B_{n}(\xi)|^2 - |A_{n}(\xi)|^2 = \omega_n$ on $\mathbb{T} $,
 \item $ f (\alpha_i) = A_n/B_n(\alpha_i) = B_n^*/A_n^*(\alpha_i)$, for $1\le i \le n+1$,
\end{enumerate}
where
\[
 \omega_n = \prod_{k=0}^{n}{(1-|\gamma_k|^2)}.
\]
\end{corollary}
The proof is by taking the determinant in \eqref{recuPQ}.  

\begin{proposition}[{\cite{Khrushchev2001}, Lemma 4.5}]
For $n \ge 0$, we have
\begin{enumerate}
\item $A_n$, $A^*_n$ and $B_n$ lie in $\cL_n$,
\item $B_n$ does not vanish on $\overline{\mathbb{D}}$, 
\item $A_n/B_n$ and $A_n^*/B_n$ are Schur functions.
\end{enumerate}
\end{proposition}

These preparations bring us to the following important
\begin{theorem}[{\cite{Khrushchev2001}, Theorem 4.6}]
\label{relation between f and An,Bn}
The Wall RFs $A_n$ and $B_n$ are connected to $f$ and $f_{n+1}$ by
\begin{equation}
\label{NP}
 f(z)= \frac{A_n(z)+\zeta_{n+1}(z) B_n^*(z) f_{n+1}(z)}{B_n(z)+ \zeta_{n+1}(z) A_n^*(z) f_{n+1}(z)}.
\end{equation}
\end{theorem}
The theorem  shows that, in Nevanlinna's 
parametrization of all Schur interpolants to $f$ 
at $(\alpha_k)_{1\le k\le n+1}$
\cite[Ch. IV, Lemma 6.1]{Garnett}, the value zero for the parameter 
yields $R_n=A_n/B_n$ while the value $f_{n+1}$ yields $f$.

\section{ORFs and Geronimus' theorem}\label{s3}
The results of this section are borrowed from \cite{Bultheel1999, Bultheel2006}. 
We formulate the results and briefly discuss them for the completeness of presentation; the proofs are generically omitted.

\subsection{Orthogonal rational functions}\label{ss31}

Let $\mu$ be a positive probability measure on $\TT$ with infinite support. 
Obviously,  $\cL_n$ is a (closed) subspace of $L^2(\mu)$, and, following \cite[Ch. 3]{Bultheel1999}, 
we regard it as a reproducing kernel Hilbert space. The reproducing kernels for $\cL_n$ are easily seen to satisfy the so-called Christoffel-Darboux relations, which can be interpreted as recurrence relations for the ORFs $(\phi_n)$. Namely, we have

\begin{theorem}[{\cite{Bultheel1999}, Theorem 4.1.1}]
\label{t0001}
For $n\geq1$, it holds that
\[
\left[
        \begin{array}{c}
        \phi_n(z)\\
        \phi_n^*(z)
        \end{array}
\right]
=
T_n(z)
\left[
        \begin{array}{c}
        \phi_{n-1}(z)\\
        \phi_{n-1}^*(z)
        \end{array}
\right],
\]
where
\begin{eqnarray}\label{ttilden}
        T_n(z)&=&\sqrt{\frac{1-|\alpha_n|^2}{1-|\alpha_{n-1}|^2}}\frac{1}
        {\sqrt{1-|\tilde{\gamma}_{n}|^2}}
        \frac{1-\bar\alpha_{n-1} z}{1-\bar\alpha_n z} 
        \left[
                \begin{array}{cc}
                1 & -\overline{\tilde{\gamma}_{n}}\\
                -\tilde{\gamma}_{n} & 1
                \end{array}
        \right]\\
&&      \times\,\left[
                \begin{array}{cc}
                \lambda_n &  0 \\
                0 &\bar \lambda_n
                \end{array}
        \right]
                \left[
                \begin{array}{cc}
                \zeta_{n-1}(z)&0\\
                0&1
                \end{array}
        \right],\nonumber
\end{eqnarray}
and
\begin{eqnarray}\label{lambdan}
\tilde{\gamma}_n=-\frac{\overline{\phi_{n}(\alpha_{n-1})}}{\overline{\phi_{n}^*(\alpha_{n-1})}}, &&
\eta_n =\frac{1-\alpha_n\bar\alpha_{n-1}}{1-\bar\alpha_n\alpha_{n-1}} ,\\
&&\lambda_n = \frac{|1-\bar\alpha_n\alpha_{n-1}|}{1-\alpha_n\bar\alpha_{n-1}}
        \frac{\overline{\phi_{n}^*(\alpha_{n-1})}}{|\phi_{n}^*(\alpha_{n-1})|}  \frac{\overline{\kappa_{n-1}}}{|\kappa_{n-1}|} \, \eta_n. \label{etan1}
\end{eqnarray}
\label{recu ORF first kind}
\end{theorem}

\begin{definition}\label{d02}
We call $\tilde{\gamma}_{n}$, given by \eqref{lambdan},  
the $n$-th {\it Geronimus parameter} of the measure $\mu$ 
(with respect to the sequence $(\alpha_k)$).
\end{definition}
Corollary 3.1.4 from \cite{Bultheel1999} says that for $z \in\mathbb{D},\ n \ge 1$
\begin{equation}\label{e001}
\phi_n^*(z) \neq 0,\quad \ |\phi_n(z)/\phi_n^*(z)|<1,
\end{equation}
and, consequently, $\tilde{\gamma}_{n}$ is 
well-defined and that $|\tilde\gamma_n|<1$.

We will normalize $\phi_n$ by setting $\lambda_n=1$, see \eqref{etan1}.
{\it Thus from now on, $\phi_n$ is the orthogonal rational function of degree
$n$ satisfying:}
\begin{equation}
\lambda_n=\frac{1- \alpha_n \bar\alpha_{n-1}}{|1-\alpha_n \bar \alpha_{n-1}|}
        \frac{\overline{\phi_{n}^*(\alpha_{n-1})}}{|\phi_{n}^*(\alpha_{n-1})|}   \frac{\overline{\kappa_{n-1}}}{|\kappa_{n-1}|}=1.
\label{normalisation phin}
\end{equation}
This normalization is from Langer-Lasarow \cite{Langer}. 
It differs from the one made in Bultheel et al. \cite{Bultheel1999},  
that corresponds to $\kappa_n = \ovl{\phi_n^*(\alpha_n) }>0$.
However, in the classical case, $\alpha_n\equiv 0$, it is easily checked by induction, that it  also matches the normalization $k_n>0$ made in (\ref{e041}). 

Relations \eqref{e001} mean that the roots of $\phi_n$ lie 
in $\overline{\mathbb{D}}$. Theorem \ref{t0001} implies that
the roots of the orthogonal rational functions $\phi_n$ are, in fact, in $\mathbb{D}$.  
Another useful fact is that the OR-functions $(\phi_k)_{0\le k\le n},$ are orthonormal in $L^2\left( \frac{P(.,\alpha_n)}{|\phi_n|^{2}} dm \right)$, see \cite[Theorem 6.1.9]{Bultheel1999}.

We already saw a definition of ORFs of the second kind (see the discussion 
following Definition \ref{d01}). Presently, the OR-functions of the second kind  
will be introduced by an explicit formula:
 \begin{equation}\label{e1}
 \left\{
\begin{array} {l}
 \psi_0=1, \\
 \psi_n(z) = \displaystyle{\int_{\mathbb{T}} \dfrac{t+z}{t-z} \left( \phi_n (t) - \phi_n (z) \right) d\mu(t)}.
\end{array}
\right. 
\end{equation}
Both definitions turn out to be equivalent (see Theorem \ref{Th Geronimus} or \cite[Theorem 6.2.5]{Bultheel1999}), but the 
one above is better suited for computations.  The next result wraps Lemmas 4.2.2 and 4.2.3 from  \cite{Bultheel1999}, whose proof is a direct computation using the orthogonality of $(\phi_n)$.

\begin{lemma}
Let $n \ge 1$ and the function $g$ be so that $g_* \in \mathcal{L}_{n-1}$. Then  
\[
 \psi_n(z) g(z) = \int_{\mathbb{T}}{\frac{t+z}{t-z} \left(\phi_n(t) g(t) - \phi_n(z) g(z) \right) d\mu(t)} .
\]
Similarly, for $h$ such that  $h_* \in \zeta_n \mathcal{L}_{n-1}$, we have
\begin{equation}
\label{seck2}
 -\psi_n^*(z) h(z) = \int_{\mathbb{T}}{\frac{t+z}{t-z} \left(\phi_n^*(t) h(t) - \phi_n^*(z) h(z) \right) d\mu(t)}.
\end{equation}
\label{psin_f}
\end{lemma}

Recall the Herglotz transform $F_\mu$ of a measure $\mu$ defined in 
\eqref{e06}. Plugging
$h=(\cB_n)_*$ in \eqref{seck2} and using that $\phi_n$ is 
$\mu$-orthogonal to constants, we obtain  at once 

\begin{proposition}[{\cite{Bultheel2006}, Theorem 3.4}]
\label{F_psi_phi}
Let $\phi_n$ be the ORF of the first kind and $\psi_n$ be as in
\eqref{e1}. Then
\begin{equation}
\label{divFmu}
 F_\mu(z)=\frac{\psi_n^*(z)}{\phi_n^*(z)}+ \frac{z \mathcal{B}_n(z) u_n(z) }{\phi_n^*(z)},
\end{equation}
where $u_n$ is a analytic function in $\DD$ given by
\begin{equation}
\label{Hermitebisc}
u_n(z)=2\int_{\TT}(\phi_n)_*(t)\,\frac{d\mu(t)}{t-z},
\ \ \ \ z\in\DD.
\end{equation}
\end{proposition}
In particular, $\psi_n^*/\phi_n^*$ interpolates
$F_\mu$ at $0$ and at the $\alpha_k$ for $1\leq k\leq n$.

The theorem to come is \cite[Theorem 4.2.4]{Bultheel1999},
with a different normalization.

\begin{theorem}
The ORFs $(\phi_n)$ and the $(\psi_n)$ from \eqref{e1} together satisfy:
\begin{equation}\label{e3}
\left[
        \begin{array}{cc}
        \phi_n & \psi_n\\
        \phi_n^* & -\psi_n^*
        \end{array}
\right]
= 
\frac{\sqrt{1-|\alpha_n|^2}}{1-\bar\alpha_n z}\frac{1}{\Pi_n}
\left(
\prod_{k=n}^{1} 
\left[
        \begin{array}{cc}
        1 & -\overline{\tilde{\gamma_k}}\\      
        -\tilde{\gamma_k} & 1
        \end{array}
\right]
\left[
        \begin{array}{cc} 
        \zeta_{k-1}(z)&0\\
        0&1
        \end{array}
\right]
\right)
\left[
        \begin{array}{cc}
        1 & 1\\
        1 & -1
        \end{array}
\right],
\end{equation}
where $\Pi_n=\prod_{k=n}^{k=1} \sqrt{1-|\tilde{\gamma_k}|^2}$. In particular, $\psi_n$ is in $\mathcal{L}_n$.
\label{recurrence_phin_psin}
\end{theorem}

By taking determinants in \eqref{e3}, we get for $z\in\DD$
\[
\phi_n(z) \psi_n^*(z) + \phi_n^*(z) \psi_n(z) = 2 \frac{1-|\alpha_n|^2}{(1-\overline{\alpha_n} z)(z-\alpha_n)} z \mathcal{B}_n(z),
\]
and, consequently, for $z\in\TT$,
\begin{equation}\label{ORF&Poisson}
  \phi_n(z) \psi_n^*(z) + \phi_n^*(z) \psi_n(z) = 2 \mathcal{B}_n(z) P(z,\alpha_n).
\end{equation}
 
\subsection{Geronimus theorem}\label{ss41}
The Geronimus-type theorem below is central for the whole construction. It seems first stated in \cite{Langer}, but it is implicitly contained in \cite[Sect. 6.4]{Bultheel1999}. 
\begin{theorem}
Let $f\in\cS$ and $\mu$ be the measure associated to $f$ by \eqref{e06}.
Then, for $k\ge0$,
\[
 \tilde{\gamma}_{k+1}=\gamma_k,
\]
where $(\tilde{\gamma_k})$ are the Geronimus parameters defined in \eqref{lambdan}
and $\left(\gamma_k\right)$  the Schur parameters defined in \eqref{e02}.
\label{Th Geronimus}
\end{theorem}
Thus the Geronimus parameters  and the Schur parameters of of a measure $\mu$
coincide.   It follows from \eqref{e3} that $(\psi_n)$ 
meets the same recurrence relations
as $(\phi_n)$ only with Geronimus parameters $-\tilde{\gamma}_{n}$ 
rather than $\tilde{\gamma}_{n}$. Thus, we see that the definition of
the ORFs of the second kind given in \eqref{e1}
coincides with the one made in the introduction.

\begin{proof}
The idea is to compare the recurrence formulas \eqref{recuPQ} and \eqref{e3}. 
We assume the sequence $(\alpha_k)$ is simple, {\it i.e. } $\alpha_k\not =\alpha_j$ for $k\not=j$. The proof in the general case follows by a limiting argument.
By \eqref{e3},  we have
\begin{eqnarray*}
\lefteqn{
\left[
        \begin{array}{cr}
        \phi_{n+1}(z) & \psi_{n+1}(z)\\
        \phi_{n+1}^*(z)& -\psi_{n+1}^*(z)
        \end{array}
\right]
}\\
&=&
\Delta_{n+1}
\left(
\prod_{k=n+1}^{k=1}
\left[
        \begin{array}{cc}
         -1 &  0 \\
        0 & 1
        \end{array}
\right]
\left[
        \begin{array}{cc}
        1 & \overline{\tilde{\gamma_k}}\\       
        \tilde{\gamma_k} & 1
        \end{array}
\right]
\left[
        \begin{array}{cc} 
        \zeta_{k-1}(z)&0\\
        0&1
        \end{array}
\right]
\left[
        \begin{array}{cc}
         -1 &  0 \\
        0 & 1
        \end{array}
\right]
\right)
\left[
        \begin{array}{cr}
        1 & 1\\
        1& -1
        \end{array}
\right]
\end{eqnarray*}
where
\[
\Delta_{n+1}=
\frac{\sqrt{1-|\alpha_{n+1}|^2}}{1-\bar\alpha_{n+1} z}\frac{1}{\prod_{k=1}^{n+1} \sqrt{1-|\tilde{\gamma_k}|^2}}.
\]
Let now $U_n/V_n$ be the $n$-th convergent of the Schur function with Schur parameters
$\gamma_k :=\tilde{\gamma}_{k+1},\ k \ge 0$.
Proposition \ref{recurrence_An_Bn} provides us with the following expression for $\phi_n$, $\psi_n$: 
\begin{eqnarray*}
&&\left[
        \begin{array}{cr}
        \phi_{n+1}(z) & \psi_{n+1}(z)\\
        \phi_{n+1}^*(z)& -\psi_{n+1}^*(z)
        \end{array}
\right]\\
&&=
\Delta_{n+1}
\left[
        \begin{array}{cc}
         -1 &  0 \\
        0 & 1
        \end{array}
\right]
\left[
        \begin{array}{ll}
        V_n^* & U_n^*\\
        U_n & V_n
        \end{array}
\right]
\left[
        \begin{array}{ll}
        \zeta_0 & 0\\
        0 & 1
        \end{array}
\right]
\left[
        \begin{array}{cc}
         -1 &  0 \\
        0 & 1
        \end{array}
\right]
\left[
        \begin{array}{cr}
        1 & 1\\
        1& -1
        \end{array}
\right]\\
&&= 
\Delta_{n+1}
\left[
        \begin{array}{ll}
        z V_n^* - U_n^* & z V_n^* + U_n^*\\
        -z U_n + V_n & -z U_n - V_n
        \end{array}
\right].
\end{eqnarray*}
Therefore,
\begin{equation}
\label{relation Psi_n A_n0}
\begin{split}
\lefteqn{
\left[
        \begin{array}{cr}
        \phi_{n+1}(z) & \psi_{n+1}(z)\\
        \phi_{n+1}^*(z)& -\psi_{n+1}^*(z)
        \end{array}
\right]
}
\\
& = 
\frac{\sqrt{1-|\alpha_{n+1}|^2}}{1-\bar\alpha_{n+1} z}\frac{1 }{\prod_{k=1}^{n+1} \sqrt{1-|\tilde{\gamma_k}|^2}}
\left[
        \begin{array}{ll}
        z V_n^* - U_n^* & z V_n^* + U_n^*\\
        -z U_n + V_n & -z U_n - V_n
        \end{array}
\right],
\end{split}
\end{equation}
and
\begin{equation}
 \frac{\psi_{n+1}^*}{\phi_{n+1}^*}=\frac{1 + z \frac{U_n}{V_n}}{1 - z \frac{U_n}{V_n}}.
\label{interpol_psi_phi_A_n_B_n0}
\end{equation}
Consequently,
\[
 \frac{U_n(z)}{V_n(z)}=\Omega_z\left(\frac{\psi_{n+1}^*(z)}{\phi_{n+1}^*(z)}\right),
\]
where $\Omega_z(w)=(w-1)/(z(w+1))$.
From Proposition \ref{F_psi_phi}, we get 
\[
F(\alpha_{j+1})=\left({\frac{\psi_{n+1}^*}{\phi_{n+1}^*}}\right)(\alpha_{j+1}).
\]
Recalling that $f(z)=\Omega_z(F(z))$, it follows by Proposition 
\ref{gamma_values_f} that the $n+1$ first Schur parameters of the function 
$U_n/V_n$  and 
of the function $f$ coincide.  
\end{proof}

The theorem shows that the functions $U_n$ and $V_n$ are equal to the
WRFs  $A_n$ and $B_n$ corresponding to $f$. In particular, 
(\ref{relation Psi_n A_n0}) and (\ref{interpol_psi_phi_A_n_B_n0}) imply
\begin{equation}
\label{relation Psi_n A_n}
\begin{split}
\lefteqn{
\left[
        \begin{array}{cr}
        \phi_{n+1}(z) & \psi_{n+1}(z)\\
        \phi_{n+1}^*(z)& -\psi_{n+1}^*(z)
        \end{array}
\right]
}
\\
& = 
\frac{\sqrt{1-|\alpha_{n+1}|^2}}{1-\bar\alpha_{n+1} z}\frac{1 }{\prod_{k=1}^{n+1} \sqrt{1-|\tilde{\gamma_k}|^2}}
\left[
        \begin{array}{ll}
        z B_n^* - A_n^* & z B_n^* + A_n^*\\
        -z A_n + B_n & -z A_n - B_n
        \end{array}
\right]
\end{split}
\end{equation}
and
\begin{equation}
 \frac{\psi_{n+1}^*}{\phi_{n+1}^*}=\frac{1 + z \frac{A_n}{B_n}}{1 - z \frac{A_n}{B_n}}.
\label{interpol_psi_phi_A_n_B_n}
\end{equation}

\subsection{Consequences of Geronimus theorem}\label{ss34}
Arguing as in  \cite[Corollary 5.2]{Khrushchev2001}, we readily see that the Schur function $A_n/B_n$ corresponds to the measure $\frac{P(.,\alpha_{n+1})} {|\phi_{n+1}|^{2}} dm$. 

The next theorem provides one with  a helpful relation between the density $\mu'$ of the absolutely continuous part of $\mu$,  the Schur remainders $(f_n)$, and the  ORFs $(\phi_n)$. It is a counterpart to Theorem 2 from  \cite{Khrushchev2001}, see also \cite{NjaVel}. Since it is heavily used in the sequel, we give the proof.

\begin{theorem}
Let $(\phi_n)$ and $(f_n)$ be the ORFs and Schur remainders associated 
to $\mu$ and $f$,   respectively. Then it holds  a.e. on $\TT$ that
\[
\mu' = \frac{1-|f_n|^2}{|1- \zeta_n \frac{\phi_n}{\phi_n^*} f_n|^2} \frac{P(.,\alpha_n)}{|\phi_n|^2}. 
\]
\label{Psi_n sigma'}
\end{theorem}

\begin{proof}
From Theorem \ref{relation between f and An,Bn}, we have on $\mathbb{T}$
\begin{eqnarray}
        1-|f|^2 & = & 1-\left|\frac{A_n + \zeta_{n+1} B_n^* f_{n+1}}{B_n +\zeta_{n+1} A_n^* f_{n+1}} \right|^2 \nonumber \\
        & = & \frac{|B_n+\zeta_{n+1}A_n^* f_{n+1}|^2-|A_n+\zeta_{n+1}B_n^* f_{n+1}|^2}{|B_n+\zeta_{n+1}A_n^* f_{n+1}|^2}.
                \label{deuxieme ligne imbouf}
\end{eqnarray}
Notice that $A_n^* \overline{B_n} = \overline{A_n} B_n^*$ on $\mathbb{T}$, so that 
\[
\zeta_{n+1} A_n^* f_{n+1} \overline{B_n} + B_n \overline{\zeta_{n+1} A_n^* f_{n+1}} 
- \overline{A_n} \zeta_{n+1} B_n^* f_{n+1} - A_n \overline{\zeta_{n+1} B_n^* f_{n+1}} = 0.  
\]
Therefore, on expanding (\ref{deuxieme ligne imbouf}) and recalling  Corollary \ref{det_A_B_on_T}, we find that 
\begin{equation*}
1-|f|^2= \frac{(|B_n|^2-|A_n|^2) (1-|f_{n+1}|^2)}{|B_n+\zeta_{n+1}A_n^* f_{n+1}|^2}
=\frac{\omega_n (1-|f_{n+1}|^2)}{|B_n+\zeta_{n+1}A_n^* f_{n+1}|^2},
\end{equation*}
where $\omega_n = \prod_{k=0}^n{(1-|\gamma_k|^2)}$.
 
Again by Theorem \ref{relation between f and An,Bn}, we obtain
\begin{eqnarray*}
|1-z f|^2 & = & \left| 1 - \frac{z A_n + \zeta_{n+1} z B_n^* f_{n+1}}{B_n + \zeta_{n+1} A_n^* f_{n+1}} \right|^2 \\
& = & \left| \frac{B_n - z A_n + \zeta_{n+1} f_{n+1} (A_n^*-z B_n^*) }{B_n + \zeta_{n+1} A_n^* f_{n+1}} \right|^2.
\end{eqnarray*}
On the other hand,  Theorem \ref{Th Geronimus} and (\ref{relation Psi_n A_n}) show
\[
        \left\{
        \begin{array}{ll}
        z B_n^* - A_n^* &=  \frac{1-\bar\alpha_{n+1} z}{\sqrt{1-|\alpha_{n+1}|^2}} {\sqrt{\omega_n}} \phi_{n+1} \\
        B_n - z A_n &= \frac{1-\bar\alpha_{n+1} z}{\sqrt{1-|\alpha_{n+1}|^2}}{\sqrt{\omega_n}} \phi^*_{n+1}
        \end{array}
        \right.
\] 
and therefore
$$
|1-z f|^2= \omega_n \frac{|1-\bar \alpha_{n+1} z|^2}{1-|\alpha_{n+1}|^2} \left| \frac{\phi_{n+1}^* - \zeta_{n+1} f_{n+1}  \phi_{n+1}}{B_n + \zeta_{n+1} A_n^* f_{n+1}} \right|^2.
$$
Recall that $\mu'(\xi)=(1-|f(\xi)|^2)/(|1-\xi f(\xi)|^2)$ a.e. on $\mathbb{T}$. 
Combining all this, we obtain 
\[
\mu' = \frac{1-|f_{n+1}|^2}{|\phi_{n+1}|^2 |1- \zeta_{n+1} \frac{\phi_{n+1}}{\phi_{n+1}^*} f_{n+1}|^2} \frac{1-|\alpha_{n+1}|^2}{|\xi-\alpha_{n+1}|^2}
\] 
which achieves the proof.\end{proof}

\section{Weighted $L^2$-convergence of Schur remainders}\label{s5}
The material reviewed so far is known, and it is meant as a preparation for the forthcoming results which are new.
As we start doing analysis rather than algebra, the assumptions 
\eqref{e04} and \eqref{e08}-\eqref{e083} or \eqref{e08p}, \eqref{e082p},
will start playing a key role.

We begin quoting a lemma which is \cite[Theorem 9.7.1]{Bultheel1999}.
\begin{lemma} \label{l1}Assuming \eqref{e04}, we get in the weak-* convergence of measures
 \[
(*)-\lim_n \frac{P(.,\alpha_n)} {|\phi_n|^{2}} dm = d\mu.
\] 
\end{lemma}

The two theorems below address the $L^2$-convergence of Schur 
remainders under different assumptions. Recall  
$Acc\,(\alpha_k)$ is  the set of accumulation points of $(\alpha_k)$.  

\begin{theorem}\label{t001}
Let  \eqref{e04} be in force and $\lim_k|\alpha_k|=1$. Assume that  
\eqref{e08}-\eqref{e083} hold.  Then
\begin{equation}
\label{convL2R}
\lim_{k} \int_\TT{|f_{k}|^2 P(.,\alpha_k) dm}=0.
\end{equation}
If $(\alpha_k)$ accumulates nontangentially on
$Acc(\alpha_k)\cap\TT$, then one can replace
\eqref{e08} and \eqref{e082} with \eqref{e08p}.
\label{lim int f_n^2 cercle}
\end{theorem}

\begin{proof}
It is enough to prove \eqref{convL2R} for any subsequence $(\alpha_{n_k})$, 
converging to $\alpha\in Acc\, (\alpha_k)$. For simplicity, the subsequence is 
still denoted by $(\alpha_k)$.

By Theorem \ref{Psi_n sigma'},  we get
\[
 |\phi_n|^2 \mu' (1+|f_n|^2-2 Re(\zeta_n \frac{\phi_n}{\phi_n^*} f_n)) = (1-|f_n|^2) P(.,\alpha_n)
\]
and, consequently,
\[
 |f_n|^2=\frac{P(.,\alpha_n)-|\phi_n|^2 \mu'}{P(.,\alpha_n)+|\phi_n|^2 \mu'} + \frac{2 |\phi_n|^2 \mu'Re(\zeta_n \frac{\phi_n}{\phi_n^*} f_n)}{P(.,\alpha_n)+|\phi_n|^2 \mu'}.
\]
Hence, we obtain
\begin{equation*}
|f_{n}|^2= \frac{P(.,\alpha_n)-|\phi_n|^2 \mu'}{P(.,\alpha_n)+|\phi_n|^2 \mu'} - \frac{P(.,\alpha_n)-|\phi_n|^2 \mu'}{P(.,\alpha_n)+|\phi_n|^2 \mu'} Re\left( \zeta_n \frac{\phi_n}{\phi_n^*} f_n\right) + Re\left( \zeta_n \frac{\phi_n}{\phi_n^*} f_n\right).  
\end{equation*}
Since $ \zeta_n (\alpha_n) =0$, we get by harmonicity
\[
\int_\TT
{Re\left( \zeta_n \frac{\phi_n}{\phi_n^*} f_n\right) P(.,\alpha_n) dm} = 0,
\]
and
\[
\int_\TT
{|f_{n}|^2 P(.,\alpha_n) dm} = \int_\TT
{\frac{P(.,\alpha_n)-|\phi_n|^2 \mu'}{P(.,\alpha_n)+|\phi_n|^2 \mu'} \left(1 - Re\left( \zeta_n \frac{\phi_n}{\phi_n^*} f_n\right)\right) P(.,\alpha_n) dm}.
\] 
Obviously,
\[
 \left|1 - Re\left( \zeta_n \frac{\phi_n}{\phi_n^*} f_n\right) \right|\le 2
\]
and we get
\begin{equation}
\int_\TT{|f_{n}|^2 P(.,\alpha_n) dm} \le 2 \int_\TT{\left|1 - \frac{2|\phi_n|^2 \mu'}{P(.,\alpha_n)+|\phi_n|^2 \mu'}\right|  P(.,\alpha_n) dm}.
\label{relation int fn gn}
\end{equation} 
Let 
\begin{equation}
\label{defgn}
g_n=\frac{2|\phi_n|^2 \mu'}{P(.,\alpha_n)+|\phi_n|^2 \mu'}.
\end{equation}
Using that $4 x^2/(1+x)^2\le x$ for $ x \ge 0$,
we deduce
\begin{eqnarray*}
   \int_{\mathbb{T}} g_n^2 P(.,\alpha_n)dm & = &  \int_{\mathbb{T}} \frac{4(|\phi_n|^2 \mu'P(.,\alpha_n)^{-1})^2} {(1+|\phi_n|^2 \mu'P(.,\alpha_n)^{-1})^2} P(.,\alpha_n) dm\\
   & \le & \int_{\mathbb{T}} |\phi_n|^2 \mu' P(.,\alpha_n)^{-1} P(.,\alpha_n) dm\\
   & = &\int_{\mathbb{T}} |\phi_n|^2 \mu' dm\le \int_{\mathbb{T}} |\phi_n|^2 d\mu = 1.
 \end{eqnarray*}
Therefore, by the Schwarz inequality, it follows that 
 \begin{equation}
        \label{int g_n}
        \int_{\mathbb{T}} g_n P(.,\alpha_n) dm \le \left( \int_{\mathbb{T}} g_n^2  P(.,\alpha_n) dm \right)^{1/2} \le 1.
\end{equation}
Furthermore, again by the Schwarz inequality,
\begin{eqnarray*}
        \int_{\mathbb{T}} \sqrt{\mu^{\prime}} P(.,\alpha_n) dm & = & \int_{\mathbb{T}} \frac{\sqrt{2}|\phi_n|           \sqrt{\mu^{\prime}} \sqrt{P(.,\alpha_n)}} {\sqrt{P(.,\alpha_n) + |\phi_n|^2 \mu^{\prime}}} \frac{\sqrt{P(.,\alpha_n) + |\phi_n|^2 \mu^{\prime}} \sqrt{P(.,\alpha_n)}} {\sqrt{2}|\phi_n|} dm\\
        & \le & \left( \int_{\mathbb{T}} g_n P(.,\alpha_n) dm \right)^{1/2} \left(\frac{1}{2} \int_{\mathbb{T}} \left(\frac{P(.,\alpha_n)}{|\phi_n|^2} + \mu^{\prime}  \right) P(.,\alpha_n) dm \right)^{1/2}.
\end{eqnarray*}
Recall that  the ORFs $(\phi_k)_{0\le k\le n}$ are orthonormal in $L^2\left( \frac{P(.,\alpha_n)}{|\phi_n|^{2}} dm \right)$ and, consequently,
$$
\int_\TT f \, \frac{P(.,\alpha_n)}{|\phi_n|^{2}} dm=\int_\TT f\, d\mu
$$
for $f\in \mathcal{L}_n + \overline{\mathcal{L}_n}$. Obviously, 
$P(z,\alpha_n) = z/(z-\alpha_n) + \bar{\alpha_n} z /(1-\bar{\alpha_n} z)$ 
lies in the latter space and 
\begin{equation}
 \int_{\mathbb{T}} P(.,\alpha_n) \frac{P(.,\alpha_n)}{|\phi_n|^2} dm = \int_{\mathbb{T}} P(.,\alpha_n) d\mu.
\label{int P P phi2 dm equal int P dmu}
\end{equation}
Using  \eqref{int P P phi2 dm equal int P dmu},
we arrive at
\begin{equation}
 \int_{\mathbb{T}} \sqrt{\mu^{\prime}} P(.,\alpha_n) dm \le \left( \int_{\mathbb{T}} g_n P(.,\alpha_n) dm \right)^{1/2} \left( \int_{\mathbb{T}} P(.,\alpha_n) d\mu \right)^{1/2}.
\label{int sigma prime cercle}
\end{equation}
Recall now that $(\alpha_n)$ converges to $\alpha \in \mathbb{T}$. By hypothesis, $\mu'$ is continuous at $\alpha$ and there is no singular 
component $\mu_s$ in a neighborhood of this point. 
Thus, passing to the inferior limit in 
(\ref{int sigma prime cercle}), we obtain
\[
 \sqrt{\mu^{\prime}(\alpha)} \le \sqrt{\mu^{\prime}(\alpha)} \liminf_{n} \left( \int_{\mathbb{T}} g_n P(.,\alpha_n) dm \right)^{1/2}.
\]
Moreover, by Fatou's theorem, the same conclusion holds 
if we assume \eqref{e08p} instead of \eqref{e08}-\eqref{e082}
provided that $(\alpha_k)$ accumulates nontangentially on
$Acc(\alpha_k)\cap\TT$.
Therefore, since $\mu^\prime(\alpha)>0$,  
\[
 \liminf_{n} \int_{\mathbb{T}} g_n P(.,\alpha_n) dm \ge 1.
\]
Combining this inequality with (\ref{int g_n}), we see that
 \begin{equation}
\label{combine}
 \lim_n \int_{\mathbb{T}} g_n P(.,\alpha_n) dm = \lim_n \int_{\mathbb{T}} g_n^2  P(.,\alpha_n) dm = 1,
\end{equation}
and subsequently that
\begin{eqnarray*}
\lim_n \int_{\mathbb{T}} (1-g_n)^2 P(.,\alpha_n) dm&=&\int_{\mathbb{T}} P(.,\alpha_n) dm - 2 \lim_n \int_{\mathbb{T}} g_n  P(.,\alpha_n) dm\\
&+&\lim_n \int_{\mathbb{T}} g_n^2  P(.,\alpha_n) dm = 0.
\end{eqnarray*}
With the Schwarz inequality and (\ref{relation int fn gn}), we 
finish the proof of the first part of the theorem.
\end{proof}

\begin{remark}\label{rk31} As a partial converse, \eqref{convL2R}  implies that $|f|<1$ a.e. 
on $Acc(\alpha_k)\cap\TT$.
\end{remark}
Indeed, observe if $|f|=1$ a.e. on $E\subset Acc(\alpha_k)\cap\TT, |E|>0$, that $|f_n|=1$ a.e. on $E$ by Theorem \ref{Psi_n sigma'}. The Lebesgue's theorem says that the set of density points of $E$ coincides with $E$ up to a set of Lebesgue measure zero. Comparing the Poisson kernel and the box kernel, we see that the integral 
in \eqref{convL2R} cannot go to zero if we pick for $\alpha$ a 
density point of $E$.

A similar convergence holds when the $(\alpha_n)$ are compactly included 
in $\mathbb{D}$. The statement below may seem strange, since the 
Poisson kernel $P(.,\alpha_n)$ is bounded from above and below and therefore
superfluous. However, it is convenient to prove the theorem in this form to have it team up with Theorem \ref{t001} in order to produce Corollary \ref{c03}.

\begin{theorem}\label{t002}
Let the sequence $(\alpha_k)$ be compactly included in $\mathbb{D}$. Then,  $|f| < 1$ a.e. on $\mathbb{T}$ if and only if
\begin{equation}
\label{convrem}
\lim_{n} \int_\TT{|f_{n}|^2 P(.,\alpha_n) dm}=0.
\end{equation}
\label{lim int f_n^2 compact}
\end{theorem}
\begin{proof}
The ``if'' part is trivial since $|f_n|=1$ wherever $|f|=1$, so
we focus on the ``only if''.
As a preliminary, notice that if $I$ is an open arc on $\mathbb{T}$ 
such that $\mu$ has no mass at the end-points of $I$, 
it holds that
\begin{equation}
\limsup_n \int_{I} \frac{P(.,\alpha_n)}{ |\phi_n|^{2}} dm \le \mu(I).
\label{cvgce faible I}
\end{equation}
Indeed, in this case, any nested sequence of open arcs $I_m$ decreasing 
to $\overline{I}$ is such that $\lim_m \mu(I_m)=\mu(\overline{I})=\mu(I)$.
Therefore by the Tietze-Urysohn theorem, there is 
to  each $\varepsilon>0$ 
a non-negative function $h_I\in {\mathcal C}(\TT)$ such that 
$h_I=1$ on $\overline{I}$ and 
$\int_{\mathbb{T}} h_I d\mu \le \mu(I) + \varepsilon$.  
Obviously
\[
 \int_{I} \frac{P(.,\alpha_n)}{|\phi_n|^2} dm \le \int_{\mathbb{T}} h_I \frac{P(.,\alpha_n)}{|\phi_n|^2} dm,
\]
and  using Lemma \ref{l1}
\[
 \lim_n \int_{\mathbb{T}} h_I \frac{P(.,\alpha_n)}{|\phi_n|^2} dm = \int_{\mathbb{T}} h_I d\mu \le \mu(I) + \varepsilon.
\] 
Since $\varepsilon$ was arbitrary, this settles the preliminary.
Next, define $g_n$ as in \eqref{defgn}. Arguing as in the previous theorem, 
we see that equation (\ref{int g_n}) still holds. 
Now,  it is enough to show that the conclusion 
of the theorem holds for some infinite subsequence of each
sequence of integers. Thus, by Helly's theorem, we are left to 
establish \eqref{convrem} along a subsequence $n_k$ such that
$\alpha_{n_k}\to \alpha\in Acc(\alpha_k)$, 
$\alpha\in\DD$, and having  the property that $g_{n_k}$ converges  
to $g\in L^\infty(\TT)$ in the $*$-weak sense. Clearly $0\leq g\leq 1$ for the same is true of 
$g_{n_k}$. Pick $\xi\in\TT$ a Lebesgue point of both $g$ and $\mu$, and let
$(I_m)$ be a nested sequence of open arcs decreasing to $\{\xi\}$
such that $\mu$ has no mass at the end-points of any $I_m$. For each $m$,
by the Schwarz inequality,
\begin{eqnarray}
\frac{1}{|I_m|} \int_{I_m} \sqrt{\mu^{\prime}} dm&=&\frac{1}{|I_m|} \int_{I_m} \frac{\sqrt{2}|\phi_{n_k}|                 \sqrt{\mu^{\prime}}} {\sqrt{P(.,\alpha_{n_k}) + |\phi_{n_k}|^2 \mu^{\prime}}} \frac{\sqrt{P(.,\alpha_{n_k}) + |\phi_{n_k}|^2 \mu^{\prime}}} {\sqrt{2}|\phi_{n_k}|} dm \nonumber\\
        & \le & \left( \frac{1}{|I_m|} \int_{I_m} g_{n_k} dm \right)^{1/2} \left( \frac{1}{2|I_m|} \int_{I_m} \left(\frac{P(.,\alpha_{n_k})}{|\phi_{n_k}|^2} + \mu^{\prime}  \right) dm \right)^{1/2}.
        \label{int sigma prime}
\end{eqnarray}
Passing to the limit in (\ref{int sigma prime}) as $n_k\to\infty$
and using 
(\ref{cvgce faible I}), we obtain 
 \[
 \frac{1}{|I_m|} \int_{I_m} \sqrt{\mu^{\prime}} dm \le \left( \frac{1}{|I_m|} \int_{I_m} g dm \right)^{1/2} \left( \frac{1}{2} \frac{\mu(I_m)}{|I_m|} + \frac{1}{2|I_m|} \int_{I_m} \mu^{\prime} dm \right)^{1/2}.
 \]
Letting now $m\to\infty$ yields
 \begin{equation}
\label{supge}
 \sqrt{\mu^{\prime}(\xi)} \le \sqrt{g(\xi)} \left( \frac{1}{2} \mu^{\prime}(\xi) + \frac{1}{2} \mu^{\prime}(\xi)  \right)^{1/2} \le \sqrt{g(\xi)} \sqrt{\mu^{\prime}(\xi)}.
 \end{equation}
By Lebesgue's theorem almost every $\xi\in\mathbb{T}$ satisfies our 
requirements, and from our assumption that $|f|<1$ we have
$\mu^{\prime}>0$, a.e. on $\TT$. Consequently  $g \ge 1$ by \eqref{supge}
hence in fact $g=1$, a.e. on $\mathbb{T}$.
Recalling that $\lim_n P(.,\alpha_n)=P(.,\alpha)$ uniformly on $\mathbb{T}$,  
we obtain \eqref{combine} from (\ref{int g_n}) 
and conclude as in Theorem \ref{lim int f_n^2 cercle}. 
\end{proof}

\begin{corollary}\label{c03}
Let  \eqref{e04}, \eqref{e08}-\eqref{e083}  hold and  
$|f| < 1$ a.e. on $\TT$. Then
\[
\lim_{k} \int_\TT{|f_{k}|^2 P(.,\alpha_k) dm}=0.
\]
When $(\alpha_k)$ accumulates nontangentially on
$Acc(\alpha_k)\cap\TT$, assumptions  \eqref{e08}-\eqref{e082}
may be replaced by  \eqref{e08p}.
\label{cgce L2 disque ferme}
\end{corollary}
\begin{proof} It is readily checked that Theorems 
\ref{lim int f_n^2 cercle} and \ref{lim int f_n^2 compact} remain 
valid for subsequences. If the corollary did not hold, it
would contradict one of them.
\end{proof}

A closer look at the proof of Theorem \ref{lim int f_n^2 cercle} 
shows that assumption  \eqref{e083}
is not really necessary. If $\alpha\in Acc(\alpha_k)\cap\TT $ and 
$\lim_k\alpha_k=\alpha$, all we need is 
$$
\lim_{k} \int_\TT P(.,\alpha_k)\, d\mu_s=0.
$$
For instance if $\mu_s$ is a Dirac mass at $\alpha$ and the $\alpha_k$ 
converge tangentially to $\alpha$, this could still hold.

\section{Convergence of  Wall rational functions ${A_n}/{B_n}$}\label{s6}
We now discuss different kinds of convergence for the WRFs. 
This is essentially an interpretation of the results in the previous section,
except that we appeal at some point to Proposition \ref{bornephin} and
Theorem \ref{c04}. The reader will easily convince himself that
there is no loophole, {\it i.e.} 
that these do not use any result of the present section. 

\subsection{Convergence on compact subsets and w.r.t. pseudohyperbolic distance} 
\label{ss61}
Let us begin with an old result that goes back to \cite{Wall}.
\begin{theorem}
Let \eqref{e04} hold. Then $A_n/B_n$ converges to $f$ uniformly on compact subsets of $\mathbb{D}$.
\end{theorem}

\begin{proof}
As $(A_n/B_n)$ is a family of Schur functions, it is normal. 
Therefore a subsequence that converges uniformly on compact subsets of 
$\DD$ can be extracted from any subsequence.
Let $g$ be the limit of such a subsequence. 
As $(A_n/B_n)(\alpha_k)=f(\alpha_k)$ for all $\leq n+1$,  $f(\alpha_k) = g(\alpha_k)$ for all $k$.
So, the function $f-g\in H^{\infty}$ vanishes on $(\alpha_k)$ hence it is 
zero by assumption \eqref{e04}. Thus, $f$ is the only limit point.
\end{proof}

Recall that the pseudohyperbolic distance $\rho$ on $\mathbb{D}$ is defined by  
$\rho(z,w) = |z-w|/|1-\bar w z|$ and it is trivially invariant under M\"obius transforms of $\DD$.

\begin{theorem}\label{t01}
Under the assumptions of Corollary \ref{cgce L2 disque ferme}, it holds that
\[
\lim_{n} \int_{\mathbb{T}} \rho\left(f,\frac{A_n}{B_n}\right)^2 P(.,\alpha_{n+1}) dm = 0.
\]
\end{theorem}

\begin{proof}
The invariance of the pseudohyperbolic distance under M\"obius transforms and  relations (\ref{f tau}) and (\ref{Rn}) show that
\[
\rho\left(f,\frac{A_n}{B_n}\right) = \rho\left(\tau_0 \circ \dots \circ \tau_{n} (f_{n+1}),\tau_0 \circ \dots \circ \tau_{n} (0)\right) = \rho(f_{n+1},0) = |f_{n+1}|.
\]
Corollary \ref{cgce L2 disque ferme} finishes the proof.
\end{proof}

\subsection{Convergence w.r.t. the hyperbolic metric} 
\label{ss62}
In the disk, the hyperbolic  metric is defined by 
\begin{equation}\label{e07}
 \mathfrak{P}(z,\omega) = \log \left( \frac{1+\rho(z,\omega)}{1-\rho(z,\omega)}\right).
\end{equation}

Here is an analogue of  the ``only if" part of  Theorem 2.6 from \cite{Khrushchev2001}.  
\begin{theorem}\label{cvg hyperbolic}
Let \eqref{e04}, \eqref{e08}-\eqref{e083} be in force, 
and $\mu\in\sz$. Then
\[
 \lim_{n} \int_{\mathbb{T}} \mathfrak{P}\left(f,\frac{A_n}{B_n} \right)^2 P(.,\alpha_{n+1}) dm = 0.
\]
If $(\alpha_k)$ accumulates nontangentially on
$Acc(\alpha_k)\cap\TT$, then it is enough to assume instead of
\eqref{e08} and \eqref{e082} that \eqref{e082p} holds.
\end{theorem}

\begin{proof}
We already saw that $\rho(f,A_n/B_n)=|f_{n+1}|$ whence
\begin{equation}
\mathfrak{P}\left(f,\frac{A_n}{B_n} \right)=\log \left( \frac{1+|f_{n+1}|}{1-|f_{n+1}|}\right). 
\label{Poincare f convergent}
\end{equation}
By Theorem \ref{Psi_n sigma'},  
\begin{equation}
|\phi_n^*|^2 |S|^2 \frac{|1-\bar \alpha_n \xi|^2}{1-|\alpha_n|^2}= \frac{1-|f_n|^2}{|1-\zeta_n \frac{\phi_n}{\phi_n^*} f_n|^2}, 
\label{presque Szego}
\end{equation}
a.e. on $\mathbb{T}$.
If $g$ is a Schur function, then $1-g\in H^\infty$  and 
$\mathrm{Re}\, (1-g)>0$, therefore $1-g$ is an outer function 
in $H^\infty(\DD)$ (see \cite{Garnett}, Corollary 4.8). Consequently,
\[
\int_{\mathbb{T}} \log |1-g|^2 P(.,\alpha_n) dm= \log|1-g(\alpha_n)|^2,
\]
and,  putting $g=\zeta_n \frac{\phi_n}{\phi_n^*} f_n$, we get
\[
\int_{\mathbb{T}} \log |1-\zeta_n \frac{\phi_n}{\phi_n^*} f_n|^2 P(.,\alpha_n) dm=0.
\]
Using the previous equality and (\ref{presque Szego}), we see that
\[
 \int_{\mathbb{T}} \log\left(|\phi_n^*|^2 |S|^2 \frac{|1-\bar \alpha_n \xi|^2}{1-|\alpha_n|^2}\right) P(\xi,\alpha_n) dm(\xi) =
\int_{\mathbb{T}} \log(1-|f_n|^2) P(\xi,\alpha_n) dm(\xi).
\]
Since $\log|\phi_n^*|$, $\log|S|$, and $\log|1-\bar \alpha_n \xi|$ are 
harmonic in $\DD$, we continue as
\begin{equation}\label{e7}
 \log \bigl(|\phi_n^*(\alpha_n)|^2 |S(\alpha_n)|^2 (1-|\alpha_n|^2)\bigr) = 
\int_{\mathbb{T}} \log(1-|f_n|^2) P(.,\alpha_n) dm,
\end{equation}
and, by Theorem \ref{c04} to come, we deduce that
\begin{equation}
 \lim_n \int_{\mathbb{T}} \log(1-|f_n|^2) P(.,\alpha_n) dm=0.
\label{lim int log 1-|f_n|^2}
\end{equation}
Since $\log(1+x) \le x$ for $x>-1$, we have
$$ 
0 \le |f_n|^2 \le -\log(1-|f_n|^2), \quad 0 \le \log(1+|f_n|) \le |f_n|. 
\label{e55}
$$ 
Therefore, by the first inequality above and (\ref{lim int log 1-|f_n|^2}),
\[
 \lim_n \int_{\mathbb{T}} |f_n|^2 P(.,\alpha_n) dm=0.
\]
From this, with the help of the second inequality and the Schwarz inequality,
\[
 \lim_n \int_{\mathbb{T}} \log(1+|f_n|) P(.,\alpha_n) dm = 0.
\]
Since $\log(1-|f_n|^2)=\log(1-|f_n|) + \log(1+|f_n|)$, we now see that
\[
 \lim_n \int_{\mathbb{T}} \log(1-|f_n|) P(.,\alpha_n) dm =0.
\]
Referring to (\ref{Poincare f convergent}), we finish the proof. 
\end{proof}

\subsection{Convergence in $L^2(\mathbb{T})$} 
The next theorem follows easily from Corollary 
\ref{cgce L2 disque ferme}. We begin with 
\begin{lemma}
\label{majfmAB}
 For $z \in \mathbb{T}$, we have
\[
\left| f(z) - \frac{A_n}{B_n}(z)\right|= |f_{n+1}(z)| \left| 1 - \frac{A_n}{B_n}(z) \overline{f(z)} \right| \leq 2 |f_{n+1}(z)|.
\]
\end{lemma}
\begin{proof}
The equality follows from $\xi_{n+1}f_{n+1}=(A_n-B_nf)/(fA_n^*-B_n^*)$ which is
inverse to \eqref{NP}. The inequality
follows because $f$ and $A_n/B_n$ are Schur.
\end{proof}
\begin{theorem}
\label{thmfABp}
The limiting relation
\begin{equation}\label{e6}
\lim_n\int_\TT \left|f-\frac{A_n}{B_n}\right|^p P(.,\alpha_{n+1})\, dm=0
\end{equation}
holds in the following cases:
\begin{enumerate}
\item if $p\ge2$ under the assumptions of Corollary \ref{cgce L2 disque ferme},
\item for $1\leq p<\infty$ if
$Acc\,(\alpha_k)\cap\TT=\emptyset$  and $|f|<1$ a.e. on $\TT$.
\end{enumerate}
\end{theorem}
\begin{proof} This is immediate from Lemma \ref{majfmAB},
Corollary \ref{cgce L2 disque ferme},
the fact that $|f_n|\leq1$, the existence of pointwise a.e. 
converging subsequences in $L^2$-convergent sequences, 
and the dominated convergence theorem.
\end{proof}
\subsection{Uniform convergence} 
When $\mu$ is sufficiently smooth, the previous $L^p$-convergence 
is uniform.  

\begin{theorem}
\label{convpAB}
Let \eqref{e04} hold and $d\mu=\mu'dm$ be absolutely continuous 
on $\TT$. Assume that $\mu'\in W^{1-1/p,p}(\TT)$ 
with $p>4$.
If $\mu'>0$ on some neighborhood $\cO(Acc(\alpha_k)\cap\TT)$, 
then 
\begin{equation}\label{einfini}
\lim_n\left\|\left(f-\frac{A_n}{B_n}\right) \sqrt{P(.,\alpha_{n+1})}\right\|_{\infty}=0.
\end{equation}
\end{theorem}
\begin{proof}
 From \eqref{e06} and \eqref{interpol_psi_phi_A_n_B_n},
one easily computes
\[f(z)-A_n/B_n(z)=2\frac{F_\mu\phi_{n+1}^*(z)-\psi_{n+1}^*(z)}
{z(1+F_\mu(z))(\psi_{n+1}^*(z)+\phi_{n+1}^*(z))}.
\]
Besides, we observe from \eqref{ORF&Poisson} that
\[\left|\phi_{n+1}^*(z)+\psi_{n+1}^*(z)\right|^2=
\left|\phi_{n+1}^*(z)\right|^2+\left|\psi_{n+1}^*(z)\right|^2+2P(z,\alpha_{n+1}),\ \ \ z\in\TT.
\]
From this and the fact that $\mbox{Re}F_\mu\geq0$, we obtain on $\TT$ that
\[|(f-A_n/B_n)\sqrt{P(.,\alpha_{n+1})}|\leq
\sqrt2\,|F_\mu\phi_{n+1}^*-\psi_{n+1}^*|.
\]
Observing that $\mu'$ is continuous on $\TT$ since $p>4$,
we may apply Corollary
\ref{c02} to the effect that $F_\mu\phi_n^*-\psi_n^*$ converges to zero 
uniformly on $\TT$.
\end{proof}

\section{A Szeg\H{o}-type problem}\label{s7}
In this final section, we study the asymptotic behavior of ORFs 
looking for an analogue of the Szeg\H o theorem in the rational setting
when $(\alpha_k)$ may approach $\TT$. The results are of a novel type 
and, we hope, worthy for themselves. We need them also to complete the proofs
of Theorems \ref{cvg hyperbolic} and  \ref{convpAB}.

\subsection{Preliminaries} 
\label{ss71}
With $\pi_n$ defined as in (\ref{Ln}), we denote by 
$\mathcal{P}_n\left(d\mu/|\pi_n|^2 \right)$ $\subset L^2\left(d\mu/|\pi_n|^2\right)$  the subspace of  polynomials of degree at most $n$. The space $H^2\left(d\mu/|\pi_n|^2\right)$ is the closure of all polynomials in  $L^2\left(d\mu/|\pi_n|^2\right)$. The reproducing kernels of  
$H^2\left(d\mu/|\pi_n|^2\right)$ and $\mathcal{P}_n\left(d\mu/|\pi_n|^2 \right)$, are denoted by $E_n$ and $R_n$, respectively. 
Since there will be several measures involved, 
we indicate the dependence in square brackets when necessary. \
For example, we may write $\phi_n[\mu], E_n[\mu], R_n[\mu]$, 
or $S[\mu]$, see \eqref{e081}.
We also put $d\mu_n:=d\mu/|\pi_n|^2$.
\begin{proposition}
Let $\mu\in\sz$ be absolutely continuous. Then, 
\[
E_n[\mu](\xi,\omega)=\frac{1} {1-\xi \bar \omega} \frac{\pi_n(\xi) \overline{\pi_n(\omega)}} {S(\xi) \overline{S(\omega)}}.
\]
\label{E_n}
\end{proposition}
The proof is straightforward and stems from the density of polynomials 
in $H^2\left(d\mu/|\pi_n|^2\right)$, the Cauchy formula,
and the identity $|S|^2=\mu'$ on $\TT$.
\begin{proposition}
The following identity holds:
\begin{equation}\label{e61}
\left |\pi_n \phi_n^* \right|= \frac{\left| R_n(.,\alpha_n) \right|} {\|R_n(.,\alpha_n) \|_{L^2\left( d\mu_n \right)}}.
\end{equation}
\label{pi_n phi_n R_n}
\end{proposition}

\begin{proof}
 Let $p_{n-1}$ be a polynomial of degree at most $n-1$. As $\phi_n$ is orthogonal to $\mathcal{L}_{n-1}$, we have
\[
 \int_{\mathbb{T}} \overline{\phi_n} \frac{p_{n-1}}{\pi_{n-1}} d\mu = 0.
\]
On the other hand, since $\overline{\phi_n}=(\phi_n)_*$ and $1/t=\bar{t}$ on $\TT$,
\begin{eqnarray*}
 \int_{\mathbb{T}} \overline{\phi_n} \frac{p_{n-1}}{\pi_{n-1}} d\mu & = & \int_{\mathbb{T}} \phi_n^*(t)\frac{\pi_n(t)}{t^n \overline{\pi_n(t)}} \frac{p_{n-1}(t) (1-\bar \alpha_n t)}{\pi_{n}(t)} d\mu(t)\\
& = & \int_{\mathbb{T}} \pi_n(t) \phi_n^*(t) \bar{t}^{n-1} p_{n-1}(t) (\bar t-\bar \alpha_n) \frac{d\mu(t)}{|\pi_n(t)|^2}\\
& = & \int_{\mathbb{T}} \pi_n(t) \phi_n^*(t) \overline{\left( t^{n-1} \overline{p_{n-1}\left(\frac{1}{\bar t}\right)} (t- \alpha_n) \right)} \frac{d\mu(t)}{|\pi_n(t)|^2}.
\end{eqnarray*}
As $t^{n-1} \overline{p_{n-1}\left(1/{\bar t}\right)}$ 
ranges over $\mathcal{P}_{n-1}(z)$ when $p_{n-1}$ does, 
$\pi_n \phi_n^*$ is $\mu_n$-ortho\-go\-nal to every polynomial of degree 
$\leq n$ that vanishes at $\alpha_n$. This is also true of 
$R_n(.,\alpha_n)$, hence $\pi_n \phi_n^*$ and  $R_n(.,\alpha_n)$ are proportional. Since the 
right-hand side  of \eqref{e61} and $\pi_n \phi^*_n$ have unit 
norm in $L^2(d\mu_n)$, we are done.  
\end{proof}

In the following corollary $\mu$ is not necessarily absolutely continuous.
\begin{corollary}
For $\mu\in\sz$ and $n \ge 1$, we have that
\begin{equation}
|\phi_n^*(\alpha_n)|^2 |S(\alpha_n)|^2 (1-|\alpha_n|^2) = \frac{R_n(\alpha_n,\alpha_n)}{E_n[\mu_{ac}](\alpha_n,\alpha_n)} \le 1.
\label{quantite ratio Rn En}
\end{equation}
\label{borne1}
\end{corollary}

\begin{proof}
By elementary properties of reproducing kernels, we get
\[\| R_n(.,\alpha_n) \|^2_{L^2(d\mu_n)}= R_n (\alpha_n,\alpha_n),
\ \ \mbox{and}\ \ 
\| E_n(.,\alpha_n) \|^2_{L^2(d\mu_n)}= E_n (\alpha_n,\alpha_n).
\]
Therefore, from Proposition \ref{pi_n phi_n R_n}, we obtain
\[
 |\pi_n(\alpha_n) \phi_n^* (\alpha_n)|^2 = \frac{|R_n (\alpha_n,\alpha_n)|^2} {\|R_n (.,\alpha_n) \|^2_{L^2\left( d\mu_n\right)}} = R_n(\alpha_n,\alpha_n)
\]
hence the equality in (\ref{quantite ratio Rn En}) from the 
formula for $E_n[\mu_{ac}]$ in Proposition \ref{E_n}.

Observing now that $||.||_{L^2(d\mu_{ac})}\le ||.||_{L^2(d\mu)}$, we get 
a contractive injection
\[H^2\left(d\mu/|\pi_n|^2\right)\subset H^2\left(d\mu_{ac}/|\pi_n|^2\right),\]
from which it follows easily that
$E_n[\mu](w,w)\le E_n[\mu_{ac}](w,w)$, for $w\in\DD$.

Since $R_n (.,\alpha_n)$ is the orthogonal projection of 
$E_n[\mu] (.,\alpha_n)$ on 
$\mathcal{P}_n\left(d\mu/|\pi_n|^2\right)$
\[
\| R_n (.,\alpha_n)\|^2_{L^2(d\mu_{n})} \le \|E_n[\mu] (.,\alpha_n) \|^2_{L^2(d\mu_{n})},
\]
and therefore 
$$
\frac{R_n(\alpha_n,\alpha_n)}{E_n[\mu_{ac}](\alpha_n,\alpha_n)}\le
\frac{R_n(\alpha_n,\alpha_n)}{E_n[\mu](\alpha_n,\alpha_n)}\le 1,
$$
as desired.
\end{proof}
It is a well-known theorem of Beurling that \cite[Ch. II, Theorem 7.1]{Garnett}
that functions of the form $Sp$, with $p$ a polynomial, are dense in 
$H^2(\DD)$ when $S$ is outer. We shall need a local refinement of this 
result (compare to \cite[Ch. 2, Theorem 7.4]{Garnett}) 
where, in addition, polynomials get replaced by functions in $\cL_n$.
 
\begin{lemma}
\label{lemmix}
Let \eqref{e04} hold and $\cO$ be open in $\TT$. 
Let $S\in H^2(\DD)$ be an outer function which is continuous
on $\cO$ with $|S|>\delta>0$ there. Then,
to every compact $K\subset {\mathcal O}$, there is a sequence of rational 
functions $R_{m}\in\cL_{m}$ such that
\begin{itemize}
\item[\it (i) \ ] $\|1-R_mS\|\to 0$ as $m\to\infty$,
\item[\it (ii) ]  the functions $1-R_mS$ go to zero uniformly on $K$.
\end{itemize}
\end{lemma}

\begin{proof} 
Recall that $\log|S|\in L^1(\TT)$ and
put $u_n=\min\{a_n, -\log|S|\}$, where $a_n>0$ tends to 
$+\infty$ so fast that 
\begin{equation}
\label{serc}
\sum_{n=0}^\infty \left(1-\exp\left(\int_{\TT}(u_n+\log|S|)\,dm\right)
\right)<\infty.
\end{equation}
Let $S_n$ be the outer function such that $|S_n|=e^{u_n}$ on $\TT$,
normalized so that $S_n(0)>0$. Then $S_n\in H^\infty(\DD)$
and $|S_nS |\leq 1$ on $\TT$ with
$|S_nS|=1$  on $\mathcal{O}$ for $n$ large enough, therefore we can write
\begin{equation}
\label{exev}
SS_n(z)=\exp \left(\int_{\mathbb{T}\setminus\mathcal{O}} 
\frac{t+z}{t-z}\log|SS_n|\, dm(t) \right), 
\end{equation}
showing that $SS_n$ extends analytically across $\mathcal{O}$ 
to an analytic function $G_n$
on $\overline{\CC}\setminus(\TT\setminus\mathcal{O})$. Moreover,
$(G_n)$ is a normal family since $|\log|S_n||\leq |\log|S||$ on $\TT$.
Besides, expanding $\|1-S_nS\|^2$ and using (\ref{serc}), we obtain
\begin{equation}
\label{BorelC}
\sum_{n=0}^\infty \|1-S_nS\|^2\leq 2\,
\sum_{n=0}^\infty (1-S_n(0)S(0))<\infty
\end{equation}
so that, by the Borel-Cantelli lemma, $SS_n$ converges to 1
a.e. on $\TT$. Then, by normality, $SS_n$ converges to 
1 locally uniformly on $\mathcal{O}$. 

Next, fix a compact $K\subset \cO$ and
let $S_{n,r}(z):=S_n(rz)$ for $0<r<1$.
As $S_{n,r}=P_{rz}*S_n$, and
since $S_n\in L^\infty(\TT)$ is
continuous on  ${\mathcal O}$ where it equals $G_n/S$, 
it follows from standard properties of 
Poisson integrals \cite[Ch. 2]{Garnett} that $S_{n,r}$ converges to $S_n$ 
boundedly pointwise a.e. on $\TT$ and locally uniformly on 
${\mathcal O}$ as $r\to1$.  In particular, $S_{n,r}S$ converges 
to $S_nS$ in $L^2(\TT)$ for fixed $n$ as $r\to1$.
Hence to each $n$ there is $r_n$ such that, say,
\[\left\{
\begin{array}{lcl}
\|1-S_{n,r_n}S\|&<&\|1-S_{n}S\|+2^{-n},\\
\sup_K|S_{n,r_n}-S_n|&<&1/n.
\end{array}
\right.
\]
Clearly  $S_{n,r_n}$ lies in $A(\mathbb{D})$, therefore is can be uniformly
approximated on $\TT$ by functions from $\cup_k\cL_k$ since 
\eqref{e04} holds. Therefore, to each $n$, there is an integer $m_n$ and
$R_{m_n}\in\cL_{m_n}$ such that
\begin{equation}
\label{nums}
\left\{
\begin{array}{lcl}
\|1-R_{m_n}S\|&<&\|1-S_nS\|+2^{-n},\\
\sup_K|R_{m_n}-S_n|&<&1/n.
\end{array}
\right.
\end{equation}
Without loss of generality, we assume that $m_n$ strictly increases 
with $n$. Now, by \eqref{BorelC} and since $|SS_n|\in L^\infty(\TT)$, 
the
first relation in \eqref{nums} implies 
\[\sum_{n=0}^\infty \|1-R_{m_n}S\|^2<\infty\]
whence $R_{m_n}S$ converges to 1 in $H^2(\DD)$ as $n\to\infty$.
In another connection,
\[|1-R_{m_n}S|\leq |1-S_nS|+|R_{m_n}-S_n||S|\]
and the second relation in \eqref{nums} yields that
 $R_{m_n}S$ converges  uniformly to 1 on $K$ when $m_n\to\infty$.
To complete the proof, it remains  to put 
$R_m=R_{m_{k}}$ where $k$ is the greatest integer such that 
$m_k\leq m$.
\end{proof}

\subsection{An a priori bound on ORFs} 
We derive in this subsection {\it a priori} estimates for ORFs 
akin to the classical bounds for orthogonal polynomials
\cite[Ch. 4, Theorems 4.6, 4.8]{Geronimus1}, \cite[Ch. 12, Theorem 12.1.3]{Szego}.
These in fact are new even in the classical polynomial case,
as they yield information in cases where $\mu'$ vanishes,
thereby generalizing some of the results from \cite{Rakhmanov}.
Their proof rely on basic properties of the Sobolev spaces $W^{1,p}(\Omega)$. 
For $1<p<\infty$ and $\Omega\subset\CC$ an open set with
boundary $\partial\Omega$, recall that
$$
W^{1,p}(\Omega)=\{f\in L^p(\Omega) : ||f||_{L^p(\Omega)}+||f'||_{L^p(\Omega)}
<\infty\},
$$
where the derivatives are understood in the distributional sense.
If $\partial\Omega$ is piecewise smooth
and $\mathcal D$ indicates the space
of $\mathcal{C}^\infty$ 
functions with compact support in $\CC$, then
the restriction  $\mathcal D|_{\Omega}$  is dense in $W^{1,p}(\Omega)$.
For $g\in W^{1,p}(\Omega)$ and $(g_n)$ a sequence in $\mathcal D|_{\Omega}$
converging to 
$g$, one can show that the trace of $g_n$ on $\partial\Omega$ 
converges in  $W^{1-1/p,p}(\Omega)$, see \eqref{defW1-1/p}.
This allows one to define the 
\emph{trace} of $g\in W^{1,p}(\Omega)$ on $\partial\Omega$  as a member of 
$W^{1-1/p,p}(\partial\Omega)$. With this definition,
Stokes' formula holds for Sobolev differential forms 
just like it does for smooth ones. 

We put $\eta=x+iy$ and use the standard notation
\[\frac{\partial}{\partial \eta}=\frac{1}{2}\left(\frac{\partial}{\partial x}
-i\frac{\partial}{\partial y}\right),\qquad
\frac{\partial}{\partial {\bar \eta}}=\frac{1}{2}\left(
\frac{\partial}{\partial x}+i\frac{\partial}{\partial y}\right).\]
The usual rules of differentiation apply to $\partial/\partial\eta$,
$\partial/\partial{\bar \eta}$, and the
relation $\partial V/\partial\bar{\eta}=0$ means 
that $V$ is analytic.  
We need a function-theoretic lemma. 
\begin{lemma}
\label{SobSmu}
Let $I\subset\TT$ be an open arc and $\Omega\subset \DD$ an open set such that
$\overline{\Omega}\cap \TT\subset I$. If $g\in H^1(\DD)$ is 
such that $g|_{I}\in W^{1-1/p,p}(I)$ for some $1<p<\infty$,
then $g|_{\Omega}\in W^{1,p}(\Omega)$.
\end{lemma}
\begin{proof}
The restriction $g|_{I}$ extends to a function
$h\in W^{1-1/p,p}(\TT)$ \cite[7.69]{Adams}.
By standard elliptic regularity,
there is a harmonic function $U\in W^{1,p}(\DD)$ such that 
$U|_{\TT}=h$, where the trace is understood in the 
Sobolev sense
\cite{campanato}. Any harmonic conjugate $V$ of $U$
in turn belongs to $W^{1,p}(\DD)$ for 
$\partial V/\partial{\bar \eta}= i\partial U/\partial{\bar \eta}$
by the Cauchy-Riemann equations. Hence the analytic function
$F=U+iV$ lies in $W^{1,p}(\DD)$, and it follows from the 
definition that $F(r\eta)\to F(\eta)$ in $W^{1,p}(\DD)$ as $r\to1^-$.
Thus by the trace theorem, the restriction of $F$
to every circle $\TT_r$ centered at 0 of radius $r<1$ has 
$W^{1-1/p,p}(\TT_r)$ norm at most $C \|F\|_{W^{1,p}(\DD)}$ for 
some constant $C$ \cite[7.39]{Adams}.
{\it A fortiori} then, $F\in H^p(\DD)$ and its Sobolev trace on $\TT$ 
must coincide
with its nontangential limit.
Consequently  $g-F$ is a $H^1(\DD)$-function which is pure imaginary
on $I$, therefore it extends analytically by reflection across $I$.
In particular $g-F$ is smooth on a neighborhood of $\overline{\Omega}$
and $g=F+(g-F)$ lies in $W^{1,p}(\Omega)$.
\end{proof}

Our {\it a priori} bound will depend on the connection between $\phi_n$, $\psi_n$
and $u_n$ defined in \eqref{Hermitebisc}. By Proposition \ref{F_psi_phi},
$zu_n(z)=F_\mu(z)(\phi_n(z))_*-(\psi_n(z))_*$ for $z\in\TT$.
Multiplying by $\phi_n$ and taking 
real parts,  we get for $z\in\TT$
\[\mu'(z)|\phi_n(z)|^2={\rm Re}((\psi_n)_*(z)\phi_n(z))+{\rm Re}
(z\phi_n(z)u_n(z)),
\]
and a short computation using \eqref{ORF&Poisson} gives us
\[\left|\mu'(z)\phi_n(z)-\frac{z{\bar u}_n(z)}{2}\right|^2=
\mu'(z)P(z,\alpha_n) + \left|\frac{u_n(z)}{2}\right|^2.
\]
Thus, either $|\mu'(z)\phi_n(z)|\leq |u_n(z)|$ or 
$|\mu'(z)\phi_n(z)|^2/4<\mu'(z)P(z,\alpha_n)+|u_n(z)|^2/4$.
Therefore, for $z\in \TT$,
\begin{equation}
\label{inegmuu}
\mu'^2(z)|\phi_n(z)|^2\leq |u_n(z)|^2+\mu'(z) P(z,\alpha_n).
\end{equation}
\begin{proposition}\label{bornephin}
Let $\mu\in\sz$ and $I\subset \TT$ be an open arc disjoint from $\supp\mu_s$.
Assume that $S\in W^{1-1/p,p}(I)$ with $4/3<p<\infty$.
Then, to each compact $K\subset I$, 
\begin{itemize}
\item[\it i)]
there is a neighborhood ${\mathcal O}(K)$ in $I$
such that ${u_n}|_{{{\mathcal O}(K)}} \in W^{1-1/\gamma,\gamma}(\mathcal{O})$ 
for $1<\gamma<4p/(p+4)$,
with  norm depending on $\mu$, $K$, and $\gamma$ only.
\item[\it ii) ]  If moreover $p>4$, then $|u_n|\leq C$ on $\mathcal{O}(K)$
and $u_n\in H_s(\mathcal{O}(K))$ for $0<s<(p-4)/2p$,
where $C$ and the H\"older constant depend 
on $\mu$, $K$, and $s$ only; in particular, from \eqref{inegmuu}, we obtain for $\xi\in K$
\begin{equation}
\label{majphinp}
\mu'^2(\xi)|\phi_n(\xi)|^2\leq C+\mu' P(\xi,\alpha_n).
\end{equation}
\end{itemize}
\end{proposition}

\begin{proof}
We may assume that $K\neq\TT$, otherwise 
the conclusion follows upon writing $\TT=K_1\cup K_2$.
Let $J=(e^{i\theta_1},e^{i\theta_2})$ be an open arc
compactly included in $I$ and containing  $K$.
Fix $0<\varepsilon<1$ and consider the radial segments
 $c_1:=[e^{i\theta_1},(1+\varepsilon)e^{i\theta_1}]$, $c_2:= [(1+\varepsilon)e^{i\theta_2},e^{i\theta_2}]$, and the circular arc
$c_3=\{(1+\varepsilon)e^{i\theta} : \theta_1\le \theta\le \theta_2\}$.
Let $\mathcal{C}=c_1\cup c_3\cup c_2$ be the open contour joining 
$e^{i\theta_1}$ to $e^{i\theta_2}$.  
Orient the Jordan curve
$\Gamma=\mathcal{C}\cup J$ counterclockwise, and let $\Omega_1$ denote 
its interior. Put  
$\Omega=\{z\in\DD;\,1/\bar{z}\in\Omega_1\}$ for the reflected set.
Lemma \ref{SobSmu} implies that $S\in W^{1,p}(\Omega)$, 
hence $G(z):=S(1/{\bar z})$ and $H(z):=\overline{S(1/{\bar z})}$ 
belong to $W^{1,p}(\Omega_1)$.
By a classical estimate 
\cite[Theorem 5.4]{Duren},  $H^2(\DD)$ embeds continuously in 
$L^\beta(\DD)$ for $2<\beta<4$. {\it A fortiori} then, $S\in L^\beta(\Omega)$ 
whence $G,H\in L^\beta(\Omega_1)$ by reflection. In particular,
from the Leibnitz rule and H\"older's inequality, it follows since $p>4/3$
that $GH(z)=|S(1/{\bar z})|^2$ lies in 
$W^{1,\alpha}(\Omega_1)$ for some $\alpha>1$.
Pick $z\notin\overline{\Omega_1}$ and
apply Stokes' theorem on $\Gamma$ to the differential form 
$GH(\eta)(\phi_n)_*(\eta)/(\eta(\eta-z))\,d\eta$:
\begin{equation}\label{e101}
\int_{\mathcal{C}\cup J} 
\frac{GH(\eta)(\phi_n)_*(\eta)}{\eta}\,\frac{d\eta}{\eta-z}
=-\int_{\Omega_1} \frac{(\partial G/\partial{\bar \eta})(\eta)}
{\eta-z}
\frac{H(\phi_n)_*(\eta)}{\eta}
\,d\eta\wedge d{\bar{\eta}},
\end{equation}
where we took into account that $H(\phi_n)_*(\eta)/(\eta(\eta-z))$ 
is analytic on $\Omega_1$, since $H$ and $(\phi_n)_*$ are analytic on 
$\overline{\CC}\setminus\overline{\DD}$ while $0,z\notin \overline{\Omega_1}$.

As $GH(\xi)d\xi/\xi=id\mu(\xi)$ on $\bar J$ because  
$\supp\mu_s\cap I=\emptyset$ 
by assumption,   
we deduce from \eqref{e101} and \eqref{Hermitebisc} that, for $z\in\DD$
\begin{equation}
\label{ecartsing}
\begin{array}{lll}
u_n(z)&=&{\displaystyle 2\int_{\TT\setminus J}
(\phi_n)_*(\xi)\,\frac{d\mu(\xi)}{\xi-z}
-{2i}\int_{\mathcal{C}}
\frac{GH(\xi)(\phi_n)_*(\xi)}{\xi}\,\frac{d\xi}{\xi-z}}\\
&-&{\displaystyle {2i}
\int_{\Omega_1} \frac{(\partial G/\partial{\bar \eta})(\eta)}
{\eta-z}
\frac{H(\phi_n)_*(\eta)}{\eta}
\,d\eta\wedge d{\bar{\eta}}.}
\end{array}
\end{equation}
 
On any $\cO(K)$ where $z$ remains at strictly positive 
distance from  $\TT\setminus J$,
the first integral in the right-hand side of \eqref{ecartsing}
is uniformly bounded and smooth 
by the Schwarz inequality because $||(\phi_n)_*||_{L^2(d\mu)}=1$
(recall that $|(\phi_n)_*|=|\phi_n|$ on $\TT$).

Next, since $\phi_nS\in H^2(\DD)$ and
$||\phi_n S||= ||\phi_n||_{L^2(d\mu_{ac})}\leq 1$, it follows from the 
Fej\`er-Riesz inequality \cite[Theorem 3.13]{Duren} that the
$L^2$-norm of $\phi_n S$ over any diameter of $\DD$ is at most $1/\sqrt 2$.
Also, the $L^2$-norm of $\phi_nS$ over the circle centered at zero
of radius $1/(1+\varepsilon)$ is less than 1. 
Hence, by reflection across $\TT$, the $L^2$-norm of $H(\phi_n)_*$ on 
$\mathcal{C}$ is uniformly bounded.
Moreover, since $G\in W^{1,p}(\Omega_1)$, the trace theorem implies
that its restriction to $\mathcal{C}$
lies in $W^{1-1/p,p}(\mathcal{C})$.
By the embedding theorem for Besov spaces \cite[Theorem 7.34]{Adams},
this restriction  belongs\footnote{It belongs in fact to the Lorentz space
$L^{p/(2-p),p}(\mathcal{C})$ that we did note introduce; the latter is 
included in $L^{p/(2-p)}(\mathcal{C})$ 
because $p<p/(2-p)$ since $p>1$, see
\cite[Lemma 1.8.13]{Ziemer}.}
to $L^{p/(2-p)}(\mathcal{C})$ 
if $p<2$, to each $L^q(\mathcal{C})$ with $1\leq q<\infty$  if $p=2$,
and it is bounded if $p>2$. 
Thus from  H\"older's inequality, it follows
since $p>4/3$ that the $L^2(\mathcal{C})$-norm of
$G|_{\mathcal{C}}$ it at most $C\|G\|_{W^{1,p}(\Omega_1)}$, where $C$
is a constant depending only of $p$ and $\mathcal{C}$.
Consequently, by the Cauchy-Schwarz inequality,
the second integral in the right-hand side of (\ref{ecartsing})
is  uniformly bounded and smooth on any  $\cO(K)$ remaining 
at positive distance from $\mathcal{C}$.

We turn to the third integral, which is
taken with respect to two-dimensional Lebesgue measure
since $d\eta\wedge d{\bar \eta}=-2idxdy$.  For  $z\in\Omega_1$, observe that  the function
\[V(z)=\int_{\Omega_1}\frac{v(\eta)}{\eta-z}d\eta\wedge d{\bar \eta}
\]
lies in $W^{1,\gamma}(\Omega_1)$ 
whenever $v\in L^\gamma(\Omega_1)$ with $1<\gamma<\infty$. 
Indeed, if $d$ is the diameter of $\Omega_1$, it holds that
$\|V\|_{L^\gamma(\Omega_1)}\leq 6d\|v\|_{L^\gamma(\Omega_1)}$ 
\cite[Theorem 4.3.12]{Astala} while the distributional derivatives
$\partial V/\partial \bar{z}$ and $\partial V/\partial z$ equal respectively
to $v$ and the restriction to $\Omega_1$ of $\mathcal{S}\check{v}$, 
where $\check{v}$ is the extension of $v$ by 0 to the whole of $\CC$ and
$\mathcal{S}$ indicates the 
Beurling transform \cite[Theorem 4.3.10]{Astala}. Since the latter is 
a bounded operator on $L^\gamma(\CC)$ \cite[Theorem 4.5.3]{Astala}, it follows 
that $V\in W^{1,\gamma}(\Omega_1)$ with norm depending on $\Omega_1$ and
$\|v\|_{L^\gamma(\Omega_1)}$.
Apply this to
\[v(\eta)=\frac{H(\eta)(\phi_n)_*(\eta)}{\eta}
\, \frac{\partial G}{\partial{\bar \eta}}(\eta)\]
so that $V$ becomes 
the third integral in \eqref{ecartsing}, up to the factor $-2i$.
On the one hand, $G\in W^{1,p}(\Omega_1)$ so that
$\partial G/\partial{\bar \eta}\in L^p(\Omega_1)$.
On the other hand, we pointed out already that 
$H^2(\DD)$ embeds in $L^\beta(\DD)$ for $2<\beta<4$, hence
the $L^\beta(\Omega)$-norm of $\phi_nS$ is uniformly bounded and so is
the $L^\beta(\Omega_1)$-norm of $H(\phi_n)_*$ by reflection.
Therefore by
H\"older's inequality, $v\in L^\gamma(\Omega_1)$ with
$1/\gamma=1/p+1/\beta$. 
As $p>4/3$, this allows us to pick $\gamma$ arbitrarily in the range
$(1,4p/(p+4))$,
and assertion $i)$ now follows from the 
trace theorem.

If $p>4$ we can pick $2<\gamma<4p/(p+4)$, so assertion $ii)$ is a 
consequence of \eqref{inegmuu} and the fact that $W^{1-1/\gamma,\gamma}(I)$
embeds continuously in $H_{1-2/\gamma}(I)$.
\end{proof} 

The importance of the above proposition lies with the 
fact that the bounds are independent of $n$ and $(\alpha_k)$,
except for the presence of $P(.,\alpha_n)$ in \eqref{majphinp} 
(which reduces to 1 in the classical case). 

It is useful to know conditions on 
$\mu'=|S|^2$ for Proposition \ref{bornephin} to apply.
Here is a simple criterion. 
\begin{lemma}
\label{critmu}
Let $1<p<\infty$ and $\mu\in\sz$. For $I\subset\TT$ an open arc, if
$\mu'|_{I}\in W^{1-1/p,p}(I)$ and $0<\delta\leq \mu'(t)\leq M<\infty$
on $I$, it holds that $S|_{J}\in W^{1-1/p,p}(J)$ for each relatively compact 
subarc $J\subset I$.
\end{lemma}
\begin{proof}
Let $\varphi\in W^{1-1/p,p}(\TT)$ coincide with $\mu'/2$ on $I$ with
$0<\delta'\leq \varphi\leq M'<\infty$; such an extension is easily constructed
by reflexion across the endpoints of $I$, see \cite[Theorem 1.5.2.3]{Grisvard}.
As $\varphi\geq\delta'>0$, we get that $\log \varphi\in W^{1-1/p,p}(I)$ for
$\log$ is Lipschitz continuous on the range of $\varphi$. Therefore
$H=\log \varphi+i\check{(\log \varphi)} \in W^{1-1/p,p}(\TT)$,
and so does $S_1=\exp H$ because $\exp$ is Lipschitz continuous on the 
range of $H$ since $\varphi\leq M'$. Now, $S_1$ is an outer function 
having the same modulus as $S$ on $I$, therefore we see as in 
\eqref{exev} that $S/S_1$ extends analytically across $I$. This entails
that $S\in W^{1-1/p,p}(J)$ whenever $J$ is relatively compact in $I$.
\end{proof}

It is straightforward to check that the product of 
bounded $W^{1-1/p,p}(I)$-functions again lies
in $W^{1-1/p,p}(I)$. Thus if $\mu'$ satisfies the conditions of
Lemma \ref{critmu}, then the conclusion still holds for
$\mu'_1(t)=\mu'(t)\Pi_{j=1}^N|t-t_j|^{\lambda_j}$ 
where $t_1,\ldots t_N\in I$ are distinct, and either $\lambda_j\geq 2$
or $\lambda_j>2(p-1)/p$ for all $j$, because
the (normalized) outer function with modulus $|t-t_j|^{\lambda_j/2}$
is just $(t-t_j)^{\lambda_j/2}$, where the branch of the  power
$\lambda_j/2$  is positive for positive arguments.
For instance $d\mu(\xi)=|1-\xi|^2dm(\xi)$ 
provides us with an example for which 
\eqref{majphinp} holds uniformly on $\TT$ although $\mu'(1)=0$.

\subsection{Convergence of ORFs for Szeg\H o measures} 
In this subsection, we assume as always that $\mu$ is a finite and positive
measure with infinite support on $\TT$, but we
no longer require it has unit mass. The ORFs and the associated
Carath\'eodory and Szeg\H o functions are defined as before. Because 
multiplying $\mu$ by $\lambda>0$ results in
the multiplication of $\phi_n$ by $\lambda^{-1/2}$ and of $S$ 
by $\lambda^{1/2}$, the results below are invariant under such
scalings.

Observe  that, similarly to the classical situation,
%
the ORF $\phi_n$ solves
the extremal problem
\begin{equation}
\label{EP1}
\max_{\xi_n\in\cL_n,\ ||\xi_n||_\mu\le 1 }\big\{|a_{n,n}|:\ \xi_n=a_{n,n} 
\cB_n+a_{n,n-1} \cB_{n-1}+\dots+a_{n,0}\cB_0\big\}.
\end{equation}
We denote the value of the problem by $\kappa_n=\kappa_n[\mu]$,
{\it i.e.}  $\kappa_n=|\phi_n^*(\alpha_n)|$. 

The extremal property  \eqref{EP1} can also be recast as 
\begin{equation}
\label{EP2}\kappa_n^{-1}=
\min_{\xi_n\in\cL_n,\ \xi_n(\alpha_n)= 1 }\|\xi_n\|_{\mu},
\end{equation}
where the extremal value is uniquely attained at 
$\xi_n=\phi^*_n/\phi_n^*(\alpha_n)$.

The Szeg\H o-type theorem we shall prove 
deals with the asymptotic behavior of 
$\kappa_n$ as $n\to+\infty$, which entails further asymptotics for $\phi_n^*$.

The statement is as follows.
\begin{theorem}\label{c04}
Let \eqref{e04}, \eqref{e08}-\eqref{e083} be in force, 
and $\mu\in\sz$.  Then
\begin{equation}
\label{thmpS}
\lim_{n} |\phi_n^*(\alpha_n)|^2 |S(\alpha_n)|^2 (1-|\alpha_n|^2) = 1.
\end{equation}
If $(\alpha_k)$ accumulates nontangentially on
$Acc(\alpha_k)\cap\TT$, then one can replace  relations
\eqref{e08} and \eqref{e082} with \eqref{e082p}.
\end{theorem}
\noindent
 
The proof of Theorem \ref{c04}
requires several steps. We first look at  smooth 
measures.   

\begin{proposition}
\label{smoothSzego}
Assume \eqref{e04} holds and $\mu\in\sz$ is absolutely continuous. 
If there is an open  neighborhood 
$\cO$ of $Acc(\alpha_k)\cap\TT$
where $\mu'\geq\delta>0$ with
$\mu'\in W^{1-1/p,p}(I)$ for each component $I$
of  $\mathcal{O}$, $p>2$, then
\eqref{thmpS} is valid.
\end{proposition}
\begin{proof}
Since $R_n(.,\alpha_n)$ is the orthogonal projection of $E_n(.,\alpha_n)$ 
on $\mathcal{P}_n(d\mu_n)$, $R_n(.,\alpha_n)$ is a polynomial of degree at 
most $n$ and the minimum
\[
 \min_{p_n \in \mathcal{P}_n} \| E_n(.,\alpha_n) - p_n \|_{L^2(d\mu_n)} 
\]
is attained exactly for $p_n=R_n(.,\alpha_n)$.
But 
$$
\| E_n(.,\alpha_n) - p_n \|^2_{L^2(d\mu_n)} = 
\int_{\mathbb{T}} \left|\frac{1} {1-\overline{\alpha_n} t} \frac{\overline{\pi_n(\alpha_n)}} {\overline{S(\alpha_n)}} - \frac{p_n(t) S(t)}{\pi_n(t)}  \right|^2 dm(t).
$$
Hence, the polynomial $P_n$ minimizing
\begin{equation}
 \min_{p_n \in \mathcal{P}_n} \left\|\frac{1} {1-\overline{\alpha_n} t}  - \frac{p_n(t) S(t)}{\pi_n(t)}  \right\|
\label{definition Pn Szego}
\end{equation}
provides us with $R_n(.,\alpha_n)$  through the relation
\begin{equation*}
R_n(.,\alpha_n) = \frac{\overline{\pi_n(\alpha_n)}} {\overline{S(\alpha_n)}} P_n. 
\end{equation*}
In view of (\ref{quantite ratio Rn En}),  we write
\begin{equation}
|\phi_n^*(\alpha_n)|^2 |S(\alpha_n)|^2 (1-|\alpha_n|^2) =\left|\frac{P_n(\alpha_n) S(\alpha_n)}{\pi_{n-1}(\alpha_n)}\right|.
\label{quantity Szego Pn}
\end{equation}
We also have for every polynomial $p_n$
\begin{eqnarray*}
\lefteqn{
\left\| \frac{1}{1 - \bar \alpha_n t} - \frac{p_n(t) S(t)}{\pi_n(t)} \right\|^2= \left\| \left(1 - \frac{p_n(t) S(t)}{\pi_{n-1}(t)}\right) \frac{1}{t- \alpha_n} \right\|^2}\\
& =  &\left\| \left(1 - \frac{p_n(\alpha_n) S(\alpha_n)}{\pi_{n-1}(\alpha_n)}\right) \frac{1}{t- \alpha_n} + \left( \frac{p_n(\alpha_n) S(\alpha_n)}{\pi_{n-1}(\alpha_n)} - \frac{p_n(t) S(t)}{\pi_{n-1}(t)} \right) \frac{1}{t- \alpha_n}  \right\|^2.
\end{eqnarray*}
Consequently,
\begin{eqnarray}
\left\| \frac{1}{1 - \bar \alpha_n t} - \frac{p_n(t) S(t)}{\pi_n(t)} \right\|^2 
& =&\left| 1 - \frac{p_n(\alpha_n) S(\alpha_n)}{\pi_{n-1}(\alpha_n)}\right|^2 \frac{1}{1- |\alpha_n|^2} \nonumber\\
&+& \left\|\left( \frac{p_n(\alpha_n) S(\alpha_n)}{\pi_{n-1}(\alpha_n)} - \frac{p_n(t) S(t)}{\pi_{n-1}(t)} \right) \frac{1}{t- \alpha_n}  \right\|^2. \label{pb approx Szego}
\end{eqnarray}
Thus, if there is a sequence of polynomials $(p_n)$  satisfying
\begin{equation}
\left\| \frac{1}{1 - \bar \alpha_n t} - \frac{p_n(t) S(t)}{\pi_n(t)} \right\|^2=o\left(\frac{1}{1-|\alpha_n|^2} \right), 
\label{lim pb approx Szego}
\end{equation}
then  we also have  (see (\ref{definition Pn Szego}))  
\[
\left\| \frac{1}{1 - \bar \alpha_n t} - \frac{P_n(t) S(t)}{\pi_n(t)} \right\|^2=o\left(\frac{1}{1-|\alpha_n|^2} \right), 
\]
and by (\ref{pb approx Szego})  
\[
\lim_{n} \frac{P_n(\alpha_n) S(\alpha_n)}{\pi_{n-1}(\alpha_n)}=1.
\]
In this case relation (\ref{quantity Szego Pn}) gives us the desired limit
\eqref{thmpS}.

Note that $\mu'$ is bounded on each component $I$ of $\cO$, since 
$W^{1-1/p,p}(I)$ consists of continuous functions for $p>2$. 
Hence $S$ is continuous on $\cO$ by Lemma \ref{critmu},
and meets the assumptions of Lemma \ref{lemmix}.
Let $K$ be a compact neighborhood of $Acc(\alpha_k)$ included in $\mathcal{O}$
and $R_n\in\cL_n$ be the 
sequence of rational functions given by the lemma. Put $R_n=p_n/\pi_n$.
As $\|1/(1-\bar{\alpha_n}t)\|^2=1/(1-|\alpha_n|^2)$, we get  
\begin{eqnarray*}
\left\| \frac{1}{1 - \bar \alpha_n t} - \frac{p_{n-1}(t) S(t)}{\pi_n(t)} \right\|^2&\le& 
\sup_{t\in \TT\bsl K} \frac 1{|1-\overline{\alpha}_nt|^2}\  \left\|1 - \frac{p_{n-1} S}{\pi_{n-1}}\right\|^2\\
&+&\frac 1{1-|\alpha_n|^2}\ \left\|1 - \frac{p_{n-1} S}{\pi_{n-1}}\right\|^2_{L^\infty(K)}.
\end{eqnarray*}
Since $K$ is a neighborhood of $Acc(\alpha_k)$,  the above supremum is bounded and the first summand in the 
right-hand side of the equation goes to zero as $n\to\infty$
by the properties of $R_{n}$. As to the second summand, it is  $o(1/(1-|\alpha_n|^2))$ since $R_{n-1}S$ converges  to 1 uniformly on $K$.
Therefore the sequence $(p_{n-1})$ satisfies (\ref{lim pb approx Szego})
whence \eqref{thmpS} holds.
\end{proof}
The assumption on $\mu'$ was only to ensure that $S$ is continuous on $\cO$. If this is known to be
the case, it is not needed.

The next proposition will sharpen Proposition \ref{smoothSzego} in that
it allows $\mu$ to have a singular part.
We need a preparatory lemma.
 
\begin{lemma}
\label{appi}
Assume \eqref{e04} holds and let $E\subset\TT,\, |E|=0,$ have an open 
neighborhood $\mathcal{U}$ in $\CC$ such that  
$\overline{\mathcal{U}}\cap Acc(\alpha_k)=\emptyset$.
Then, to every $\varepsilon>0$, there exists an integer $n_0$ and 
$R_{n_0}\in\cL_{n_0}$ such that
\begin{itemize}
\item[\it (i)\  \  ] 
$|R_{n_0}|\le 2+\varepsilon$ on $\TT$,
\item[\it (ii) ] 
$|1-R_{n_0}|\leq \varepsilon$ on $\TT\bsl\mathcal{U}$,
\item[\it (iii)] $|R_{n_0}|\leq \varepsilon$ on $E$,
\end{itemize}
\end{lemma}
\begin{proof}
Since $E=\supp\mu_s$ has Lebesgue measure zero, 
one can find
$g\in A(\DD)$ such that $g=1$ on $E$ and $|g|<1$ on 
$\overline{\DD}\setminus E$, {\it cf.} \cite[Ch. 3, Exercise 2]{Garnett}.
Pick  $m$ so large that
$|g^m|<\varepsilon/2$ on $\overline{\DD}\setminus\mathcal{U}$.
Since \eqref{e04} holds 
and $(1-g^m)\in A(\DD)$,
we can find $n_0$ and $R_{n_0}\in\cL_{n_0}$ such that
$|1-g^m-R_{n_0}|<\varepsilon/2$ on $\overline{\DD}$. 
\end{proof}
We also take note of the identity
\begin{equation}
\label{Fmuquot}
F_\mu=S[\mu]/S[\widetilde{\mu}],
\end{equation}
where $\widetilde{\mu}$ is the Herglotz measure
of $1/F_\mu$, see \eqref{e06}. Indeed, since Carath\'e\-odory functions
are outer, both sides of \eqref{Fmuquot} 
are outer functions, positive at 0, with
equal modulus a.e. on $\TT$ as can be
readily computed from \eqref{fFF}.
\begin{proposition}
\label{smoothsingSzego}
Assumptions being as in Proposition \ref{smoothSzego}, except that $\mu$
may now have a singular part satisfying \eqref{e083}, we have that
\eqref{thmpS} holds.
\end{proposition}
\begin{proof}
In view of Corollary \ref{borne1}, all we have to prove is that
\begin{equation}
\label{ineglim}
\liminf_n (1-|\alpha_n|^2)|S(\alpha_n)|^2 \kappa^2_n[\mu]\geq1.
\end{equation}
Assume first that $\mu'\geq\delta'>0$ on an open set
$\mathcal{V}\supset\supp \mu_s$ in $\TT$
and that the restriction of $\mu'$ to each
component $I$ of $\mathcal{V}$ lies in $W^{1-1/q,q}(I)$ for some $q>4$.
Then $S$ is continuous on $\mathcal{V}$.
Thanks to \eqref{e083}, we may require in addition 
that $\overline{\mathcal{V}}\cap \cO=\emptyset$ and
also, by the compactness of $\supp \mu_s$, 
that $\mathcal{V}$ has finitely many components.
Fix a neighborhood $\mathcal{W}$ of
$\supp\mu_s$ in $\TT$ with $\overline{\mathcal{W}}\subset \mathcal{V}$.
We can apply Proposition 
\ref{bornephin} to $d\mu_{ac}$ on each component $I$ of $\mathcal{V}$
with $K=\overline{\mathcal{W}}\cap I$ and 
deduce from \eqref{majphinp}, since
$\overline{\mathcal{W}}$ remains at positive distance from
$(\alpha_k)$, that each ORF $\theta_j$
associated with $\mu_{ac}$ and with {\it any} subsequence 
$(\beta_l)$ of $(\alpha_k)$ 
is bounded by a constant $C_0$ on $\overline{\mathcal{W}}$, 
where $C_0$ is {\it independent} of $j$ and of the subsequence.
In particular, since $\mu'\in L^1(\TT)$ and $|\theta_j^*|=|\theta_j|$ on $\TT$,
to any $\varepsilon>0$ there is $\eta>0$ such that for $j\in\NN,\  (\beta_l)\subset(\alpha_k)$,
\begin{equation}
\label{abscb}
\int_{\mathcal{W}_1}|\theta_j^*|^2\mu'\,dm<\varepsilon,
\end{equation}
as soon as $\mathcal{W}_1\subset\overline{\mathcal{W}}$ has Lebesgue measure 
less than $\eta$.

Pick $\varepsilon>0$  and let $\mathcal{W}_1\subset\mathcal{W}$ be an
open neighborhood of $\supp\mu_s$ in $\TT$ 
such that $|\mathcal{W}_1|<\eta$. This is 
possible since $|\supp\mu_s|=0$. Write $\mathcal{W}_1=\mathcal{U}\cap\TT$ 
where $\mathcal{U}$ is open in $\CC$ and 
$\overline{\mathcal{U}}\cap(\alpha_k)=\emptyset$. This can be ensured
because $\overline{\mathcal{W}}\cap Acc(\alpha_k)=\emptyset$.
Apply 
Lemma \ref{appi} with $E=\supp\mu_s$,
and let $R_{n_0}\in\cL_{n_0}$ be as in the lemma.
Consider the sequence $(\theta_j)$ of  ORFs 
associated with $d\mu_{ac}$ for the truncated sequence
$\beta_l=\alpha_{l+n_0}$, $l\geq1$. Hence for $n>n_0$, we have that
$|\theta^*_{n-n_0}(\alpha_n)|=\kappa'_{n-n_0}[\mu_{ac}]$ 
where the prime in $\kappa'_{n-n_0}[\mu_{ac}]$ indicates 
that we work with the truncated sequence $(\alpha_k)_{k> n_0}$.
 By \eqref{EP2}, we get
\[\kappa_n^{-2}[\mu]\leq
\int_{\TT}\left|\frac{\theta^*_{n-n_0}R_{n_0}}{\theta^*_{n-n_0}(\alpha_n)
R_{n_0}(\alpha_n)}\right|^2\mu'(t)\,dm
+\int_{\TT}\left|\frac{\theta^*_{n-n_0}R_{n_0}}{\theta^*_{n-n_0}(\alpha_n)
R_{n_0}(\alpha_n)}\right|^2d\mu_s.
\]
On the one hand, by properties $(ii)$ and  $(iii)$ of Lemma \ref{appi}, we get
\begin{equation}
\label{equm}
\int_{\TT}\left|\frac{\theta^*_{n-n_0}R_{n_0}}{\theta^*_{n-n_0}(\alpha_n)
R_{n_0}(\alpha_n)}\right|^2d\mu_s\leq
\frac{\varepsilon^2 C_0^2}{(1-\varepsilon)^2}(\kappa'_{n-n_0}[\mu_{ac}])^{-2}.
\end{equation} 
On the other hand,  by properties $(i)$ and $(ii)$ of the same lemma,
\begin{eqnarray}
\nonumber
\int_{\TT}\left|\frac{\theta^*_{n-n_0}R_{n_0}}{\theta^*_{n-n_0}(\alpha_n)
R_{n_0}(\alpha_n)}\right |^2\mu'(t)\, dm
&\leq&\frac{(1+\varepsilon)^2}{(1-\varepsilon)^2}
\int_{\TT\setminus\mathcal{W}_1} \left|\frac{\theta^*_{n-n_0}}
{\theta^*_{n-n_0}(\alpha_n)}\right|^2\mu'(t)\, dm\\
&+& \frac{(2+\varepsilon)^2}{(1-\varepsilon)^2}
\int_{\mathcal{W}_1} \left|\frac{\theta^*_{n-n_0}}
{\theta^*_{n-n_0}(\alpha_n)}\right|^2\mu'(t)\,dm.
\nonumber
\end{eqnarray}
Expanding $(2+\varepsilon)^2$ and collecting terms, while using \eqref{abscb} 
and remembering that $\|\theta^*_{n-n_0}\|_{\mu_{ac}}=1$, we obtain
\begin{eqnarray}
\nonumber
\int_{\TT}\left|\frac{\theta^*_{n-n_0}R_{n_0}}{\theta^*_{n-n_0}(\alpha_n)
R_{n_0}(\alpha_n)}\right|^2\mu'(t)\,dm
&\leq&\frac{(1+\varepsilon)^2}{(1-\varepsilon)^2}
\left(\kappa^\prime_{n-n_0}[\mu_{ac}]\right)^{-2}\\
\label{equme}
&+& \frac{(3+2\varepsilon)\varepsilon}{(1-\varepsilon)^2}
\left(\kappa^\prime_{n-n_0}[\mu_{ac}]\right)^{-2}.
\end{eqnarray}
Since $\varepsilon$ can be made arbitrarily small, we gather from 
\eqref{equm} and \eqref{equme} that to each $\varepsilon'>0$ there is $n_0$ such that
\[\kappa_n^{-2}[\mu]\leq(1+\varepsilon')
\left(\kappa^\prime_{n-n_0}[\mu_{ac}]\right)^{-2}\]
as soon as $n>n_0$. But from Proposition \ref{smoothSzego} we know that
\[\liminf_n (1-|\alpha_n|^2)|S(\alpha_n)|^2 
\left(\kappa'_{n-n_0}[\mu_{ac}]\right)^2\geq1,\]
so we obtain \eqref{ineglim} since $\varepsilon'$ is arbitrarily small. 

Next, we remove the assumption that
$\mu'\geq\delta'>0$ on $\mathcal{V}$ with
$\mu'\in W^{1-1/q,q}$ on its components, but we suppose that
$\mu'<C<\infty$ on $\mathcal{V}$. Fix a neighborhood $\mathcal{W}$ of
$\supp\mu_s$ such that $\overline{\mathcal{W}}\subset\mathcal{V}$.
To each $\eta>0$, pick a neighborhood $\mathcal{V}_\eta\subset\mathcal{W}$ of 
$\supp\mu_s$ satisfying $|\mathcal{V}_\eta|<\eta$.
Put $d\mu_\eta=\mu'_\eta dm+d\mu_s$, where
$\mu'_\eta(t)=\mu'(t)$ for $t\not\in\mathcal{V}_\eta$, $\mu_\eta'=C$ on 
$\mathcal{V}_\eta$.  Being a positive constant on $\mathcal{V}_\eta$,
$\mu'_\eta$ certainly meets the assumptions of the preceding part of the proof,
so \eqref{ineglim} holds for $\mu_\eta$.
Clearly, from \eqref{EP1},
$\kappa_n[\mu_\eta]\le\kappa_n[\mu]$  because $\mu\le\mu_\eta$ hence
\begin{equation}
\label{inegliml}
\liminf_n (1-|\alpha_n|^2)|S[\mu_\eta](\alpha_n)|^2 \kappa^2_n[\mu]\geq1.
\end{equation}
Now, let $\mathcal{V}'$ be open in $\CC$ and contain no $\alpha_k$, with
$\mathcal{V}'\cap\TT=\mathcal{V}$. Since 
\[S[\mu](z)=S[\mu_\eta](z)\,
\exp \left(\int_{\mathcal{V_\eta}} 
\frac{t+z}{t-z}\log|\mu'/C|\, dm(t) \right), 
\]
we see by dominated convergence (remember $\log\mu'\in L^1(\TT)$) that
$S[\mu]/S[\mu_\eta]$ converges uniformly to 1 in $\DD\setminus\mathcal{V}'$
as $\eta\to0$. 
Consequently \eqref{ineglim} follows from \eqref{inegliml}.

We now address the case where $\mu'$ may be unbounded in the neighborhood 
$\mathcal{V}$ of $\supp\mu_s$ but $\mu'\ge\delta''>0$ there. 
Since $p>2$, $S[\mu]$ is continuous on 
$\cO$ by Lemma \ref{critmu}.
Observe also that $F_{\mu_{ac}}=\mu'+i\check{\mu'}$ lies in $W^{1-1/p,p}(I)$
for each component $I$ of $\cO$ and that $F_{\mu_s}$ is smooth on $I$,
hence $F_\mu=F_{\mu_{ac}}+F_{\mu_s}$ is continuous and bounded on 
some open neighborhood $\mathcal{N}$ of $Acc(\alpha_k)\cap\TT$ 
with $\overline{\mathcal N}\subset\cO$.
Moreover $|F_\mu|>\delta>0$ on $\cO$,  thus
$S[\widetilde{\mu}]$ is continuous and 
positively bounded from below on $\mathcal{N}$  by \eqref{Fmuquot}.
Similarly $|F_\mu|\geq \delta''$ a.e. on $\mathcal{V}$, thus
$\widetilde{\mu}'={\rm Re}\,1/F_\mu$ is bounded there.
In addition, $\supp\widetilde{\mu}_s\cap \mathcal{N}=\emptyset$
otherwise the $H^\infty(\DD)$-function
$e^{-1/F_\mu}$ would have a singular inner factor which is not 
analytic across $\mathcal{N}$, and its modulus could not be
continuous and nonzero on $\mathcal{N}$ whereas $|F_\mu|\geq\delta$
there \cite[Ch. II, Theorem 6.2]{Garnett}.
Therefore we can apply the previous case of the proof to
$\widetilde{\mu}$ with $\mathcal{O}$ replaced by $\mathcal{N}$; indeed,
we know that $S[\widetilde{\mu}]$ is continuous on $\mathcal{N}$
which is enough to proceed by the remark after 
Proposition \ref{smoothSzego}. 
Thus, we obtain for the ORFs of the second kind $\psi_n[\mu]$:
\begin{equation}
\label{limfpsi}
\lim_{n} |\psi_n^*[\mu](\alpha_n)|^2 \,|S[\widetilde{\mu}](\alpha_n)|^2 \,
(1-|\alpha_n|^2) = 1.
\end{equation}
Recalling from Proposition \ref{F_psi_phi} that 
$\psi^*[\mu](\alpha_n)/\phi^*[\mu](\alpha_n)=F_\mu(\alpha_n)$, 
we conclude in view of \eqref{limfpsi} and \eqref{Fmuquot} that
\eqref{thmpS} again holds.

Finally, under the sole assumptions of the proposition,
let $\mathcal{V}$ be an open neighborhood of $\supp \mu_s$ in $\TT$
such that $\overline{\mathcal{V}}\cap \cO=\emptyset$, and 
$\mathcal{V}'$ be open in $\CC$ and contain no $\alpha_k$, with
$\mathcal{V}'\cap\TT=\mathcal{V}$.
Let $\mathcal{W}$
be another neighborhood of $\supp \mu_s$ with 
$\overline{\mathcal{W}}\subset\mathcal{V}$.
Put $d\mu_\varepsilon=\mu'_\varepsilon dm+d\mu_s$ with
$\mu'_\varepsilon=\mu'+\varepsilon$ on $\mathcal{W}$, 
and $\mu'_\varepsilon=\mu'$ on $\TT\bsl\mathcal{W}$. 
By what we just proved 
\[\liminf_n (1-|\alpha_n|^2)|S[\mu_\varepsilon]
(\alpha_n)|^2 \kappa^2_n[\mu_\varepsilon]\geq1,\]
and since $\kappa_n[\mu_\varepsilon]\leq\kappa_n[\mu]$ 
(because $\mu_\varepsilon\geq\mu$) while 
\[S[\mu]/S[\mu_\varepsilon]=
\exp \left(\int_{\mathcal{W}} 
\frac{t+z}{t-z}\log|\mu'/(\mu'+\varepsilon)|\, dm(t) \right),\ \ \ 
z\in\DD,
\]
converges uniformly to 1  in $\DD\setminus\mathcal{V}'$ as $\varepsilon\to0$
by the monotone convergence of $\mu'/(\mu'+\varepsilon)$ to 1 a.e. on
$\mathcal{W}$,
we conclude that \eqref{ineglim} holds.
\end{proof}
In the course of the previous proof, we noticed that \eqref{limfpsi} 
is equivalent to \eqref{thmpS}. This is worth recording, taking into
account that $\widetilde{\widetilde{\mu}}=\mu$:
\begin{corollary}
\label{Szegot}
Let $\mu\in\sz$. Then \eqref{thmpS} holds for $\mu$ if, and only if
it holds for $\widetilde{\mu}$,
the Herglotz measure of $1/F_\mu$ (see \eqref{e06}).
\end{corollary}
There are Carath\'eodory functions, with continuous
and strictly positive real part on $\TT$, whose imaginary part is unbounded.
One example is $2+\varphi$ where $\varphi$ conformally maps $\DD$ onto
$\{z=x+iy;\ |x|<1/(1+y^2)\}$,  $\varphi(0)=0$ and $\varphi'(0)>0$,
whose imaginary part is unbounded at $\pm i$,
see \cite[Ch. III, Sect. 1]{Garnett}.  If we put
$d\mu'(t)=(2+{\rm Re}\, \varphi(t))dt$, then $2+\varphi(t)=F_\mu$ and
and $\widetilde{\mu}'=\mu'/|F_\mu|^2$ is continuous but
vanishes at $\pm i$.  Letting $(\alpha_k)$ accumulate at $\pm i$, 
Theorem \ref{c04} will apply to $\mu$
and then Corollary \ref{Szegot} will provide us with an example where 
\eqref{thmpS} holds although \eqref{e082} fails.

\begin{proof}[Proof of Theorem \ref{c04}]
Let $\mathcal{O}$ be the neighborhood of
$Acc(\alpha_k)\cap\TT$ granted by \eqref{e08}-\eqref{e083}.
Shrinking $\mathcal{O}$ if necessary, we may assume that $\mu'$ is continuous 
with $\mu'\geq\delta>0$ on a neighborhood of $\overline{\mathcal{O}}$ in $\TT$.
Pick $\varepsilon>0$ and $0<r<1$ so that
the Poisson integral $h_r(z)=P_{rz}*\mu'$ satisfies
$|h_r-\mu'|<\varepsilon$ on $\overline{\mathcal{O}}$.  
Let $\mu_\varepsilon$ have singular part $\mu_s$ and
absolutely continuous part $\mu'_{\varepsilon}\, dm$
where
$\mu_{\varepsilon}'(t)=\mu'(t)$ for $t\notin \mathcal{O}$ and 
$\mu_{\varepsilon}'(t)=h_r(t)+\varepsilon$ for $t\in \mathcal{O}$.
Then $\mu'\leq\mu_{\varepsilon}'\leq\mu'+2\varepsilon$ on $\TT$ and 
$\mu_{\varepsilon}'$ is smooth on $\mathcal{O}$.
By Proposition \ref{smoothsingSzego}, we have
\begin{equation}
\label{limeps}
\lim_{n} \kappa^2_n[\mu_{\varepsilon}]\, 
|S[\mu_{\varepsilon}](\alpha_n)|^2 \,(1-|\alpha_n|^2) = 1.
\end{equation}
Since $\kappa_n[\mu]\geq \kappa_n[\mu_{\varepsilon}]$ because
$\mu\leq\mu_{\varepsilon}$, we  deduce from (\ref{limeps}) that
\begin{eqnarray}\label{e9}
&&\liminf_n (1-|\alpha_n|^2)|S(\alpha_n)|^2 \kappa^2_n[\mu] \nonumber
\\
&&\ge\liminf_n 
\frac{|S(\alpha_n)|^2}{|S[\mu_{\varepsilon}](\alpha_n)|^2}\,
(1-|\alpha_n|^2) |S[\mu_{\varepsilon}](\alpha_n)|^2 
\kappa^2_n[\mu_\varepsilon] \label{chaine}\\
&&=\liminf_n 
\frac{|S(\alpha_n)|^2}{|S[\mu_{\varepsilon}](\alpha_n)|^2}\nonumber.
\end{eqnarray}
Recalling the inequalities on $\mu', \mu'_{\varepsilon}$ given above,
we get
\begin{equation}
\label{limSzeps}
\frac{|S(z)|}{|S[\mu_{\varepsilon}](z)|}=
\exp\bigl(P_z*\log(\mu'/\mu_{\varepsilon}')\bigr)
\ge1-2\varepsilon/\delta
\end{equation}
for $z\in\DD$, and letting $\varepsilon\to0$ we obtain \eqref{ineglim}
from \eqref{chaine}, \eqref{limSzeps}. With Corollary \ref{borne1},
we finish the proof. At last, assume that
$(\alpha_k)$ accumulates nontangentially
on $Acc(\alpha_k)\cap\TT$ and that \eqref{e08}, \eqref{e082} get replaced by
\eqref{e082p}. Then we can find a sequence of 
continuous functions $\varphi_j>0$ decreasing pointwise to $\mu'$ on 
$\mathcal{O}$.
Letting $d\mu_j=d\mu_s+\mu_j'dm$ where $\mu_j'=\varphi_j$ on $\mathcal{O}$
and $\mu_j'=\mu'$ on $\TT\setminus\mathcal{O}$, we get from the first part 
of the proof that \eqref{thmpS} holds for $\mu_j$. Since $\mu_j\geq\mu$, we
deduce as in \eqref{chaine} that for each $j$ 
\begin{equation}
\liminf_n (1-|\alpha_n|^2)|S(\alpha_n)|^2 \kappa^2_n[\mu] 
\geq\liminf_n 
\frac{|S(\alpha_n)|^2}{|S[\mu_j](\alpha_n)|^2}.
\end{equation}
Without loss of generality, we may assume that $(\alpha_n)$ converges to 
$\alpha\in \overline{\DD}$. If $\alpha\in\DD$, the conclusion follows 
from the fact that $S[\mu_j](\alpha)\to S(\alpha)$ by the monotone convergence
of $\mu'_j$ to $\mu'$. If $\alpha\in\TT$, then by Fatou's theorem
\[\lim_n 
\frac{|S(\alpha_n)|^2}{|S[\mu_j](\alpha_n)|^2}=
\frac{\mu'(\alpha)}{\mu'_j(\alpha)},
\]
which can be made arbitrarily close to 1
since $\lim_j\mu'_j(\alpha)=\mu'(\alpha)>0$.
\end{proof}

\begin{corollary}\label{c01}  Let \eqref{e04}, \eqref{e08}-\eqref{e083} 
be satisfied and $\mu\in\sz$.  Then
\begin{equation}
\label{Szegom}
\lim_n \left\| S\phi^*_n(z) -\beta_n\frac{\sqrt{1-|\alpha_n|^2}}{1-\overline\alpha_n z}\right\|=0,
\end{equation}
where the unimodular factors $\beta_n$ are defined in Theorem 3' of Section 
\ref{mainresults}.
Moreover, for any sequence $(z_n)\subset\DD$, it holds that
\begin{equation}
\label{suplim1}
\lim_{n}\left\{ \phi_n^*(z_n)S(z_n) \sqrt{1-|z_n|^2} 
-\beta_n \frac{\sqrt{1-|\alpha_n|^2}\sqrt{1-|z_n|^2}}{1-\overline{\alpha}_nz_n}
\right\}=0.
\end{equation}
If $(\alpha_k)$ accumulates nontangentially on
$Acc(\alpha_k)\cap\TT$, then it is enough to assume instead of
\eqref{e08}-\eqref{e082} that \eqref{e082p} holds.
\end{corollary}

\begin{proof}
Estimating the integral, we get
\begin{eqnarray*}
\left\| S\phi^*_n(z) -\beta_n\frac{\sqrt{1-|\alpha_n|^2}}{1-\overline\alpha_n z}\right\|^2&=&
\int_\TT \left| S\phi^*_n(z) -\beta_n\frac{\sqrt{1-|\alpha_n|^2}}{1-\overline\alpha_n z}\right|^2\,dm(z)\\\le
\|\phi_n\|^2_\mu &-&2{\rm Re}\left(\frac{\ovl\beta_n}{2i\pi}
\int_\TT S\phi^*_n(z) \frac{\sqrt{1-|\alpha_n|^2}}{z-\alpha_n }
dz\right) +1\\
&=&2(1-\sqrt{1-|\alpha_n|^2}\,|S(\alpha_n)|\,|\phi_n^*(\alpha_n)|),
\end{eqnarray*}
and Theorem \ref{c04} yields \eqref{Szegom}. Next, let us set 
\[k_{z_n}=\frac{\sqrt{1-|z_n|^2}}{1-z\overline{z}_n}, \qquad
G_n(z)=S\phi^*_n(z)-\beta_n\frac{\sqrt{1-|\alpha_n|^2}}{1-\overline{\alpha}_nz}.\]
Since $\|k_{z_n}\|=1$, the relation just proven and the Schwarz 
inequality yield $\lim_n (G_n, k_{z_n})=0$. Expanding the 
scalar product gives us \eqref{suplim1}.
\end{proof}

Under the assumptions of the theorem, its conclusions also hold for $\tilde\mu$ by Corollary \ref{Szegot}.
 
\begin{corollary}\label{c02}  Let \eqref{e04} hold and $\mu\in\sz$ meet
\eqref{e08}-\eqref{e083}.
If $I$ is an open arc on $\TT$, with 
$\overline{I}\cap\supp\mu_s=\emptyset$, such that $\mu'\in W^{1-1/p,p}(I)$
for some $p>4$,
then $F_\mu\phi_n^*-\psi_n^*$ converges to zero locally uniformly on $I$.
\end{corollary}
\begin{proof}
From Corollary \ref{c01}, limit \eqref{Szegom} holds
as well as its analogue for $\widetilde{\mu}$:
\begin{equation}
\label{Szegompsi}
 \lim_n \left\| S[\widetilde{\mu}]\psi^*_n(z) -\beta_n[\widetilde{\mu}]
\frac{\sqrt{1-|\alpha_n|^2}}{1-\overline\alpha_n z}\right\|=0.
\end{equation}
Moreover, it follows
immediately from \eqref{Fmuquot} and Proposition \ref{F_psi_phi} that 
$\beta_n[\mu]=\beta_n[\widetilde{\mu}]$. Substracting \eqref{Szegompsi}
from \eqref{Szegom} and using \eqref{Fmuquot} now gives us
\[ \lim_n \left\|S[\widetilde{\mu}](F_\mu\phi_n^*-\psi_n^*)\right\|=
\lim_n \left\|S[\mu]\phi_n^*-S[\widetilde{\mu}]\psi_n^*\right\|=0.
\]
In particular, we get that from any subsequence of 
$g_n:=F_\mu\phi_n^*-\psi_n^*$
one can extract a subsequence that converges pointwise a.e. to zero on $\TT$.
But since $g_n$ is equicontinuous on compact subsets of $I$ by 
Proposition \ref{bornephin}, $ii)$,
we deduce from Ascoli's theorem that $g_n$ converges locally 
uniformly to zero on $I$.
\end{proof}
\smallskip\noindent
{\it Acknowledgments.}  The second author is grateful to members of the
APICS team from INRIA Sophia-Antipolis for numerous invitations and 
warm hospitality.


\begin{thebibliography}{99}

\bibitem{Adams} Adams, R., Fournier J. Sobolev spaces. 
Pure and Applied Maths., vol.  140, Academic Press, 2003.

\bibitem{Akhiezer} Akhiezer, N. Theory of approximation. Dover Publications Inc., New York, 1992.  

\bibitem{AV}
Anderson, B., Vongpanitlerd, S.
Network Analysis and Synthesis - A modern system theory approach. Prentice Hall, 1973.

\bibitem{Apt02} Aptekarev, A.
Sharp constant for rational approximation of analytic functions. Math. Sb.  193 (2002), no. 1-2,  1--72.

\bibitem{AVA04}
Aptekarev A.,  Van Assche, W.
Scalar and Matrix {R}iemann-{H}ilbert approach to the strong asymptotics of {P}ad\'e approximants and complex orthogonal polynomials with varying weight.
J. Approx. Theory 129 (2004), 129--166.

\bibitem{Astala}
Astala K., Iwaniec T., Martin G. Elliptic partial differential equations and quasiconformal mappings in the plane. 
Math. Series 48, Princeton Univ. Press, 2009. 

\bibitem{BKT}
Baratchart, L., K\"ustner, R., Totik, V.
Zero distribution via orthogonality.  Ann. Inst. Fourier  55  (2005), 1455--1499.

\bibitem{BaYa3} Baratchart, L., Yattselev, M.
Convergent interpolation to Cauchy integrals over analytic arcs.
to appear in Found. Constr. Math.

\bibitem{Begehr} Begehr, H. Complex analytic methods for partial 
differential equations. World Scientific, 1994.

 \bibitem{Bultheel1999} Bultheel, A., Gonz\'alez-Vera, P., Hendriksen, E., Nj\aa stad, O. Orthogonal rational functions. Cambridge University Press, Cambridge, 1999. 

\bibitem{Bultheel2006} Bultheel, A., Gonz\'alez-Vera, P., Hendriksen, E., Nj\aa stad, O.  Orthogonal rational functions on the unit circle: from the scalar to the matrix case.  Lecture Notes in Math., vol. 1883 (2006),  187--228. 

\bibitem{campanato} Campanato, S. Elliptic systems in divergence
form, Interior regularity, Quaderni, Scuola Normale Superiore Pisa, 1980.

\bibitem{Dzrbaj} Dzrbasjan, M. Orthogonal systems of rational functions on the unit circle with given set of poles.
Dokl. Akad. Nauk. SSSR 147 (1962), 1278--1281 (Russian);
English transl. in Soviet Mat. Dokl., (1962)  no.  3, 1794--1798.

\bibitem{dk1}  Denisov, S., Kupin, S. Orthogonal polynomials and a generalized Szeg\H o condition.  C. R. Math. Acad. Sci. Paris  339  (2004),  no. 4, 241--244.

\bibitem{dk2}  Denisov, S., Kupin, S. Asymptotics of the orthogonal polynomials for the Szeg\H o class with a polynomial weight.  J. Approx. Theory  139  (2006),  no. 1-2, 8--28.

\bibitem{Duren}  Duren, P.  Theory of $H^p$ spaces. Academic Press, New York, 1970.

\bibitem{FCG}  Faurre, P., Clerget, M., Germain, F. Operateurs rationels positifs. Dunod, 1979.

\bibitem{Garnett} Garnett, J. Bounded analytic functions.  Springer, New York, 2007.

\bibitem{Geronimus} Geronimus, J. On polynomials orthogonal on the circle, on trigonometric moment-problem and on allied Carath\'eodory-Schur functions. Rec. Math. [Mat. Sbornik] N. S. 15 (1944), no. 57,  99--130.

\bibitem{Geronimus1} Geronimus, J. Orthogonal polynomials.  Consultants Bureau, New York, 1961.

\bibitem{Grisvard} Grisvard, P. Elliptic problems in nonsmooth domains.
\emph{Monographs and Studies in Mathematics} 24, Pitman, 1985.
 

\bibitem{Khrushchev2001} Khrushchev, S. Schur’s algorithm, orthogonal polynomials, and convergence of Wall's continued fractions in $L^2(T)$. J. Approx. Theory 108 (2001), no. 2, 161--248. 

\bibitem{khr2} Khrushchev, S.  Classification theorems for general orthogonal polynomials on the unit circle.  J. Approx. Theory  116  (2002),  no. 2,  268--342.

\bibitem{ks1} Killip, R., Simon, B. Sum rules for Jacobi matrices and their applications to spectral theory.  Ann. of Math. (2)  158  (2003),  no. 1, 253--321.

\bibitem{Koosis} Koosis, P.  Introduction to $H^p$ spaces.  Cambridge University Press, Cambridge, 1998.

\bibitem{kr1} Krein, M. On a generalization of some investigations of G. Szeg\H o, V. Smirnoff and A. Kolmogoroff.  C. R. (Doklady) Acad. Sci. URSS (N.S.)  46  (1945). 91--94.

\bibitem{kui}
Kuijlaars, A., McLaughlin, K., Van Assche, W., Vanlessen, M. The Riemann-Hilbert approach to strong asymptotics for orthogonal polynomials on $[-1,1]$.  Adv. Math.  188  (2004),  no. 2, 337--398.

\bibitem{Langer} Langer, H., Lasarow, A.  Solution of a multiple Nevanlinna-Pick problem via orthogonal rational functions. J. Math. Anal. Appl. 293 (2004), no. 2, 605--632. 

\bibitem{lu2}
Lubinsky, D. Universality limits in the bulk for arbitrary measures on a compact set, to appear in J. d'Analyse Math. 

\bibitem{Lunot}
Lunot, V. Rational approximation techniques and frequency design: a Zolotarev problem and the Schur algorithm, Ph.D. Thesis, University of  Provence, 2008.
 
\bibitem{McLM} McLaughlin, K., Miller, P.
The $\bar\partial$ steepest descent method for orthogonal polynomials on the real line with varying weight, submitted.

\bibitem{MFMLS} Mart\'inez-Finkelstein, A., McLaughlin, K., Saff, E.
Asymptotics of orthogonal polynomials with respect to an analytic weight with algebraic singularities on the circle. Int. Math. Res. Notices ID 91426 (2006),
1-43.

\bibitem{ma} Matthaei, Y. Microwave filters, impedance matching networks
and coupling structures. New York, Mc Graw Hill, 1965.

\bibitem{Mina} Mi$\tilde{\mbox{n}}$a-D\`iaz, E. 
An expansion for polynomials orthogonal over an analytic Jordan curve.
Comm. Math. Phys. 285 (2009), no. 3,
1109--1128.

\bibitem{nev} Nevanlinna, R. \"Uber beschr\"ankte Funktionen, die in gegebenen Punkten vorgeschriebene Werte annehemen.  Ann. Acad. Sci. Fenn. 13 (1919) no. 1.

\bibitem{Nikishin} Nikishin, E., Sorokin, V. Rational approximations and orthogonality.  Translations of AMS, vol. 92, Providence, RI, 1991.  

\bibitem{NjaVel} Njastad, O., Vel\`azquez, L. 
Wall rational functions and Khrushchev's formula for orthogonal 
rational functions. Constr. Approx. 30 (2009), no. 2, 277-297.
\bibitem{Pan} Pan, K. On the convergence of rational functions orthogonal on the unit circle. J. Comput. Appl. Math.  (1996), no. 76, 315--324.

 \bibitem{yu1} Peherstorfer, F., Yuditskii, P. Asymptotics of orthonormal polynomials in the presence of a denumerable set of mass points.  Proc. Amer. Math. Soc.  129  (2001),  no. 11, 3213--3220. 

\bibitem{Rakhmanov} Rakhmanov, E. Asymptotic properties of polynomials 
orthogonal on  the unit circle with weights not satisfying the Szeg\H o 
condition (Russian). Mat. Sb. 130 (172) (1986), no. 2, pp. 151-169, 284.

\bibitem{re1} Remling, C. The absolutely continuous spectrum of Jacobi matrices, submitted.

\bibitem{Schur} Schur, I. \"Uber potenzreihen, die im innern des einheitskreises beschr\"ankt sind. J. Reine Angew. Math. 147 (1917) 205–232. English translation in: I. Schur methods in operator theory and signal processing (Operator Theory: Adv. and Appl. 18 (1986), Birkh\"auser Verlag).

\bibitem{Simon1} Simon, B. Orthogonal polynomials on the unit circle, I, II.  AMS Colloquium Publications, vol. 54,  Providence, RI, 2005.

\bibitem{Szego} Szeg\H o, G. Orthogonal polynomials. AMS, Providence, RI, 1975. 

\bibitem{wa1} Wall, H. Continued fractions and bounded analytic functions.  Bull. Amer. Math. Soc.  50  (1944). 110--119.

\bibitem{Wall} Wall, H. Analytic theory of continued fractions. Van Nostrand, New York, 1948.

\bibitem{Ziemer} Ziemer, W. P. Weakly differentiable functions. G.T.M. 120, 
Springer, 1989.
\end{thebibliography}
\end{document}